\documentclass{amsart}
\usepackage{amsfonts, amsbsy, amsmath, amssymb}

\ifx\fiverm\undefined
  \newfont\fiverm{cmr5}
\fi

\newtheorem{thm}{Theorem}[section]
\newtheorem{lem}[thm]{Lemma}

\newtheorem{prop}[thm]{Proposition}

\newtheorem{thm-con}[thm]{Theorem-Conjecture}
\numberwithin{equation}{section}

\theoremstyle{definition}

\allowdisplaybreaks

\newcommand{\f}{\Bbb F}

\begin{document}

\title[Determination of a Class of Permutation Trinomials]{Determination of a Class of Permutation Trinomials in Characteristic Three}


\author{Xiang-dong Hou}
\address{Department of Mathematics and Statistics,
University of South Florida, Tampa, FL 33620}
\email{xhou@usf.edu}


\author{Ziran Tu*}
\address{School of Mathematics and Statistics, Henan University of Science and
Technology, Luoyang 471003, China}
\email{tuziran@yahoo.com}

\author{Xiangyong Zeng$\dagger$}
\address{Faculty of Mathematics and Statistics, Hubei Key Laboratory of Applied Mathematics, Hubei
University, Wuhan, 430062,
 China}
\email{xiangyongzeng@aliyun.com}
\thanks{$\dagger$ Research partially supported by National Natural Science Foundation of China under grant No. 61761166010.}



\begin{abstract}
Let $f(X)=X(1+aX^{q(q-1)}+bX^{2(q-1)})\in\f_{q^2}[X]$, where $a,b\in\f_{q^2}^*$. In a series of recent papers by several authors, sufficient conditions on $a$ and $b$ were found for $f$ to be a permutation polynomial (PP) of $\f_{q^2}$ and, in characteristic $2$, the sufficient conditions were shown to be necessary. In the present paper, we confirm that in characteristic 3, the sufficient conditions are also necessary. More precisely, we show that when $\text{char}\,\f_q=3$, $f$ is a PP of $\f_{q^2}$ if and only if $(ab)^q=a(b^{q+1}-a^{q+1})$ and $1-(b/a)^{q+1}$ is a square in $\f_q^*$.
\end{abstract}

\maketitle

\section{Introduction}

Let $\f_q$ be the finite field with $q$ elements and let $p=\text{char}\,\f_q$. A polynomial $f\in\f_q[X]$ is called a permutation polynomial (PP) of $\f_q$ if it induces a permutation of $\f_q$. Classifications of PPs with simple or prescribed algebraic forms are important and difficult questions. In the present paper, we consider the polynomials of the form
\begin{equation}\label{1.1}
f(X)=X(1+aX^{q(q-1)}+bX^{2(q-1)})\in\f_{q^2}[X],
\end{equation}
where $a,b\in\f_{q^2}^*$. These polynomials were studied in two recent papers by Tu, Zeng, Li and Helleseth \cite{Tu-Zeng-Li-Helleseth-FFA-2018} and by Tu and Zeng \cite{Tu-Zeng-CC}. Sufficient conditions on the coefficients were found for $f(X)$ to be a PP of $\f_{q^2}$:

\begin{thm}\label{T1.1}
$f(X)$ is a PP of $\f_{q^2}$ if $a,b\in\f_{q^2}^*$ satisfy one of the following sets of conditions according to the characteristic $p$.
\begin{itemize}
\item[(i)] \cite{Tu-Zeng-Li-Helleseth-FFA-2018} $p=2$,
\begin{equation}\label{1.2}
b(1+a^{q+1}+b^{q+1})+a^{2q}=0
\end{equation}
and
\begin{equation}\label{1.3}
\begin{cases}
\displaystyle\text{\rm Tr}_{q/2}\Bigl(1+\frac 1{a^{q+1}}\Bigr)=0&\text{if}\ b^{q+1}=1,\vspace{0.5em}\cr
\displaystyle\text{\rm Tr}_{q/2}\Bigl(\frac {b^{q+1}}{a^{q+1}}\Bigr)=0&\text{if}\ b^{q+1}\ne 1.
\end{cases}
\end{equation}

\item[(ii)] \cite{Tu-Zeng-CC} $p=3$,
\begin{equation}\label{1.4}
\begin{cases}
a^qb^q=a(b^{q+1}-a^{q+1}),\vspace{0.3em}\cr
\displaystyle 1-\Bigl(\frac ba\Bigr)^{q+1}\ \text{is a square in}\ \f_q^*.
\end{cases}
\end{equation}

\item[(iii)] \cite{Tu-Zeng-CC} $p>3$, either
\begin{equation}\label{1.5}
\begin{cases}
a^qb^q=a(b^{q+1}-a^{q+1}),\vspace{0.3em}\cr
\displaystyle 1-4\Bigl(\frac ba\Bigr)^{q+1}\ \text{is a square in}\ \f_q^*,
\end{cases}
\end{equation}
or
\begin{equation}\label{1.6}
\begin{cases}
a^{q-1}+3b=0,\vspace{0.3em}\cr
\displaystyle-3\Bigl(1-4\Bigl(\frac ba\Bigr)^{q+1}\Bigr)\ \text{is a square in}\ \f_q^*.
\end{cases}
\end{equation}
\end{itemize}
\end{thm}

Bartoli \cite{Bartoli-FFA-2018} proved that for $p=2$, the conditions in Theorem~\ref{T1.1} (i) are also necessary for $f(X)$ to be a PP of $\f_{q^2}$. For a different proof for the necessity and sufficiency of the conditions in Theorem~\ref{T1.1} (i), see \cite{Hou-arXiv1803.04071}. In the present paper, we show that for $p=3$, the conditions in Theorem~\ref{T1.1} (ii) are also necessary for $f(X)$ to be a PP of $\f_{q^2}$. To put our result in perspective, we mention that there have been numerous studies on PPs of $\f_{q^2}$ of the form
\[
f_{a,b,s_1,s_2}(X)=X(1+aX^{s_1(q-1)}+bX^{s_2(q-1)})\in\f_{q^2}[X],
\]
where $1\le s_1,s_2\le q$ \cite{Bartoli-FFA-2018, Bartoli-Giulietti-FFA-2018, Bartoli-Quoos-DCC-2018, Ding-Qu-Wang-Yuan-Yuan-SIAMJDM-2015, Gupta-Sharma-FFA-2016, Hou-arXiv1803.04071, Lee-Park-AMS-1997, Li-Helleseth-arXiv1606.03768, Li-Helleseth-CC-2017, Tu-Zeng-CC, Tu-Zeng-Li-Helleseth-FFA-2018}. For a given pair $(s_1,s_2)$, necessary and sufficient conditions on $a,b$ for $f_{a,b,s_1,s_2}$ to be a PP of $\f_{q^2}$ have been determined only in the following cases: $(s_1,s_2)=(1,2)$ and $p$ is arbitrary \cite{Hou-FFA-2015b}; $(s_1,s_2)=(-1/2,1/2)$ and $p=2$ \cite{Tu-Zeng-FFA-2018}; $(s_1,s_2)=(q,2)$ and $p=2$ \cite{Bartoli-FFA-2018, Hou-arXiv1803.04071, Tu-Zeng-Li-Helleseth-FFA-2018}.

The method of the present is similar to that of \cite{Hou-arXiv1803.04071}. Our main tools are the Hasse-Weil bound and resultants of polynomials. The paper is organized as follows:

Section 2 contains some preparatory results. It is well known that $f(X)$ in \eqref{1.1} is a PP of $\f_{q^2}$ if and only if an associated rational function of degree 3 permutes the field $\f_q$, and a theorem by K. S. Williams tells when the latter happens. The Hasse-Weil bound provides additional information which is crucial in our proof. In Section~3, we state the main theorem and lay out a proof plan consisting of three cases depending on $a$ and $n$, where $q=3^n$. The three cases are treated in Sections~4 -- 6, respectively. The basic approach in the three cases is the same: computation and analysis of resultants of relevant polynomials. However, the complexity of the computations involved increases considerably from case 1 to case 3. A few brief concluding remarks are given in Section~7. The proof (Sections~4 -- 6) produces many lengthy intermediate results, which are recorded in the Appendix. If the present paper appears in a journal in the future, it is unlikely that the material in the appendix will be included. Therefore, the appendix here serves as a resource for the readers who would like to verify the proof.

\section{Preparatory Results}

Throughout the paper, $p=\text{char}\,\f_q=3$ and $f(X)$ is the polynomial in \eqref{1.1}. Let $\mu_{q+1}=\{x\in\f_{q^2}^*:x^{q+1}=1\}$. It is well known that $f(X)$ is a PP of $\f_{q^2}$ if and only if $h(X):=X(1+aX^q+bX^2)^{q-1}$ permutes $\mu_{q+1}$ \cite{Park-Lee-BAMS-2001, Wang-LNCS-2007, Zieve-PAMS-2009}. For $x\in\mu_{q+1}$ with $1+ax^q+bx^2\ne 0$, i.e., with $bx^3+x+a\ne 0$, we have
\begin{equation}\label{2.1}
h(x)=\frac{x(1+ax^q+bx^2)^q}{1+ax^q+bx^2}=\frac{a^qx^3+x^2+b^q}{bx^3+x+a}=g(x),
\end{equation}
where
\begin{equation}\label{2.2}
g(X)=\frac{a^qX^3+X^2+b^q}{bX^3+X+a}\in\f_{q^2}[X].
\end{equation}
Therefore, $f(X)$ is a PP of $\f_{q^2}$ if and only if $bX^3+X+a$ has no root in $\mu_{q+1}$ and $g(X)$ permutes $\mu_{q+1}$.

Assume that $bX^3+X+a$ has no root in $\mu_{q+1}$, which implies that $1+a+b\ne 0$. Choose $z\in\f_{q^2}\setminus\f_q$ and let $\phi(X)=(X+z^q)/(X+z)$. Then $\phi(X)$ maps $\f_q\cup\{\infty\}$ to $\mu_{q+1}$ bijectively with $\phi(\infty)=1$. Let $\psi(X)=g(1)\phi(X)=(1+a+b)^{q-1}\phi(X)$. Then $\psi^{-1}(X)$, the compositional inverse of $\psi(X)$, maps $\mu_{q+1}$ to $\f_q\cup\{\infty\}$ bijectively with $\psi^{-1}(g(1))=\infty$. Therefore, $g(X)$ permutes $\mu_{q+1}$ if and only if $\psi^{-1}\circ g\circ\phi$ permutes $\f_q\cup\{\infty\}$, i.e., if and only if $\psi^{-1}\circ g\circ\phi$ permutes $\f_q$, that is, if and only if for each $y\in\f_q$, there is a unique $x\in\f_q$ such that
\begin{equation}\label{2.3}
g\Bigl(\frac{x+z^q}{x+z}\Bigr)=(1+a+b)^{q-1}\frac{y+z^q}{y+z}.
\end{equation}
Write
\begin{equation}\label{2.4}
g\Bigl(\frac{X+z^q}{X+z}\Bigr)=\frac{A(X)}{B(X)},
\end{equation}
where
\begin{equation}\label{2.5}
A(X)=(1+a^q+b^q)X^3+(z-z^q)X^2+(z^{2q}-z^{1+q})X+b^qz^3+a^qz^{3q}+z^{1+2q},
\end{equation}
\begin{equation}\label{2.6}
B(X)=(1+a+b)X^3+(-z+z^q)X^2+(z^2-z^{1+q})X+az^3+bz^{3q}+z^{2+q}.
\end{equation}
Combining \eqref{2.3} and \eqref{2.4} gives the following proposition.

\begin{prop}\label{P2.1}
$f(X)$ is a PP of $\f_{q^2}$ if and only if
\begin{itemize}
\item[(i)] $B(X)$ has no root in $\f_q$, and
\item[(ii)] for each $y\in\f_q$, there is a unique $x\in\f_q$ such that
\begin{equation}\label{2.7}
(1+a+b)A(x)(y+z)-(1+a+b)^qB(x)(y+z^q)=0.
\end{equation}
\end{itemize}
\end{prop}

In our proof of the main result (Theorem~\ref{T3.1}), equation~\eqref{2.7} can be simplified to the form
\begin{equation}\label{2.8}
C_3x^3+C_2(y)x^2+C_1(y)x+C_0(y)=0,
\end{equation}
where $C_0(Y), C_1(Y), C_2(Y)\in\f_q[Y]$ are of degree $\le 1$ and $C_3\in\f_q^*$. Write $c_i=C_i(y)$, $0\le i\le 3$. Assume that $c_2\ne 0$ and let $x=x_1+c_1/c_2$. Then \eqref{2.8} becomes
\begin{equation}\label{2.9}
\frac{c_3}{c_2}x_1^3+x_1^2+\frac{c_1^3c_3-c_1^2c_2^2+c_0c_2^3}{c_2^4}=0.
\end{equation}
Further assume that $c_1^3c_3-c_1^2c_2^2+c_0c_2^3\ne 0$, and let $x_2=1/x_1$. Then \eqref{2.9} becomes
\begin{equation}\label{2.10}
x_2^3+\frac{c_2^4}{c_1^3c_3-c_1^2c_2^2+c_0c_2^3}x_2+\frac{c_2^3c_3}{c_1^3c_3-c_1^2c_2^2+c_0c_2^3}=0.
\end{equation}
By a theorem of K. S. Williams \cite[Theorem~2]{Williams-JNT-1975}, \eqref{2.10} has a unique solution $x_2\in\f_q$ if and only if $-(c_1^3c_3-c_1^2c_2^2+c_0c_2^3)$ is a nonsquare in $\f_q^*$.

\begin{lem}\label{L2.2}
Assume that $q\ge 3^3$ and let $k$ be a nonsquare in $\f_q^*$. In the above notation, assume that $C_2(Y)\ne 0$ and let
\begin{equation}\label{2.11}
E(Y)=-k(C_1(Y)^3C_3-C_1(Y)^2C_2(Y)^2+C_0(Y)C_2(Y)^3)\in\f_q[Y].
\end{equation}
If for every $y\in\f_q$ with $C_2(y)E(y)\ne 0$, $E(y)$ is a square in $\f_q^*$, then there exists $D(Y)\in\f_q[Y]$ such that
\begin{equation}\label{2.12}
E(Y)=D(Y)^2.
\end{equation}
\end{lem}

\begin{proof}
We may assume that $E(Y)\ne 0$. Hence $C_2(Y)E(Y)$ has at most $5$ roots in $\f_q$. Let
\[
F(X,Y)=X^2-E(Y)
\]
and
\[
V_{\f_q^2}(F)=\{(x,y)\in\f_q^2:F(x,y)=0\}.
\]
By assumption,
\begin{equation}\label{2.13}
|V_{\f_q^2}(F)|\ge 2(q-5)+5=2q-5.
\end{equation}
We claim that $F(X,Y)$ is not irreducible over $\overline\f_q$ (the algebraic closure of $\f_q$). Otherwise, let {\tt y} be transcendental over $\f_q$ and let {\tt x} be a root of $F(X,{\tt y})$. Then $\f_q({\tt x}, {\tt y})/\f_q$ is a function field with constant field $\f_q$. Let $(E({\tt y}))_0$ and $(E({\tt y}))_\infty$ denote the zero devisor and the pole divisor of $E({\tt y})$ in the rational function field $\f_q({\tt y})$. Then $\deg (E({\tt y}))_0=\deg (E({\tt y}))_\infty=\deg E(Y)\le 4$ and $(E({\tt y}))_\infty=\deg E(Y)\cdot P_\infty$, where $P_\infty$ is the place of $\f_q({\tt y})/\f_q$ at $\infty$. By \cite[Proposition III.7.3 (c)]{Stichtenoth-1993}, the genus $g$ of $\f_q({\tt x},{\tt y})/\f_q$ satisfies
\[
g\le 1-2+\frac 12\bigl(\deg E(Y)+2-\text{gcd}(2,\deg E(Y))\bigr)\le 1.
\]
The affine curve $V_{\f_q^2}(F)$ has at most two singular points in $\f_q^2$: $(0,y)$, where $y$ is a multiple root of $E(Y)$. Let $N_1$ denote the number of degree 1 places of $\f_q({\tt x},{\tt y})/\f_q$. Then by the Hasse-Weil bound \cite[V.2.3]{Stichtenoth-1993},
\[
|V_{\f_q^2}(F)|\le N_1+2\le q+1+2gq^{1/2}+2\le q+3+2q^{1/2}<2q-5,
\]
which is a contradiction to \eqref{2.13}. Hence the claim is proved.

Write $F(X,Y)=(X+D(Y))(X-D(Y))$, where $D(Y)\in\overline\f_q[Y]$. If $D(Y)\notin\f_q[Y]$, choose $\sigma\in\text{Aut}(\overline\f_q/\f_q)$ such that $\sigma(X+D(Y))\ne X+D(Y)$. Then $\sigma(X+D(Y))=X-D(Y)$ and hence
\[
V_{\f_q^2}(F)\subset V_{\f_q^2}(X+D(Y))\cap V_{\f_q^2}(X-D(Y)).
\]
It follows that $|V_{\f_q^2}(F)|\le\deg D(Y)\le 4$, which is a contradiction to \eqref{2.13}. Thus $D(Y)\in\f_q[Y]$ and the proof of the lemma is complete.
\end{proof}

In \eqref{2.12}, write
\begin{equation}\label{2.14}
E(Y)=e_0+e_1Y+e_2Y^2+e_3Y^3+e_4Y^4
\end{equation}
and
\begin{equation}\label{2.15}
D(Y)=D_0+D_1Y+D_2Y^2,
\end{equation}
where $e_i,D_j\in\f_q$. Then \eqref{2.12} is equivalent to
\begin{equation}\label{2.16}
\begin{cases}
e_0=D_0^2,\cr
e_1=-D_0D_1,\cr
e_2=D_1^2-D_0D_2,\cr
e_3=-D_1D_2,\cr
e_4=D_2^2.
\end{cases}
\end{equation}
From \eqref{2.16}, we have
\begin{equation}\label{2.17}
\begin{cases}
e_0e_3^2-e_1^2e_4=0,\cr
e_3^3-e_1e_4^2-e_2e_3e_4=0.
\end{cases}
\end{equation}

In our notation, the resultant of two polynomials $P_1(X)$ and $P_2(X)$ is denoted by $\text{Res}(P_1,P_2;X)$. Our proof of the main result relies on the ability to compute $\text{Res}(P_1,P_2;X)$, where $P_1,P_2\in\f_q[X,Y_1,\dots,Y_m]$, and to factor $\text{Res}(P_1,P_2;X)$ in $\f_q[Y_1,\dots,Y_m]$. With computer assistance, the computation of $\text{Res}(P_1,P_2;X)$ is relatively easy regardless of the number $m$ of additional variables. However, factorization in $\f_q[Y_1,\dots,Y_m]$ is difficult; it appears that the existing symbolic computation softwares are not very effective for this question when $m\ge 3$. Some techniques in our proof are aimed at reducing the number of variables in the polynomials involved.

We alert the reader of a slight abuse of notation in the paper. We will encounter many expressions of the form $P(u)$, where $P\in \f_q[X]$ and $u$ is an element of $\f_q$ yet to be determined. Naturally, $P(u)$ is an element of $\f_q$. However, we frequently treat $u$ as an indeterminate and hence $P(u)$ is viewed as a polynomial in $u$ rather than an element of $\f_q$. It should be clear from the context which point of view is taken at the moment. This harmless abuse of notation allows us to avoid additional excessive notation.

\section{Main Result}

Our main result is the following theorem.

\begin{thm}\label{T3.1}
Assume that $q=3^n$. If $f(X)$ is a PP of $\f_{q^2}$, then the conditions in \eqref{1.4} are satisfied.
\end{thm}

We assume that $n\ge 3$ since Theorem~\ref{T3.1} is easily verified for $n\le 2$.

First observe that for $\beta\in\f_{q^2}^*$,
\begin{equation}\label{3.1}
f(\beta X)=\beta X(1+a'X^{q(q-1)}+b'X^{2(q-1)}),
\end{equation}
where $a'=a\beta^{1-q}$ and $b'=b\beta^{2(q-1)}$. Also note that $a$ and $b$ satisfy \eqref{1.4} if and only if $a'$ and $b'$ do. Therefore, when proving Theorem~\ref{T3.1}, we may replace $a$ and $b$ with $a'$ and $b'$. We will consider three cases:

\subsection*{Case 1} $a$ is a square in $\f_{q^2}$.

\subsection*{Case 2} $a$ is not a square in $\f_{q^2}$ and $n$ is even.

\subsection*{Case 3} $a$ is not a square in $\f_{q^2}$ and $n$ is odd.

\section{Proof of Theorem~\ref{T3.1}, Case 1}

Assume that $a$ is a square in $\f_{q^2}$. Let $a=\gamma^{2i}$, where $\gamma$ is a primitive element of $\f_{q^2}$. Letting $\beta=\gamma^{-i}$ in \eqref{3.1} gives $a'=a\beta^{1-q}=\gamma^{i(q+1)}\in\f_q^*$. Therefore, in this case, we may assume that $a\in\f_q^*$.

\subsection*{Case 1.1} Assume that $b\notin \f_q$. Let $b_1=(b+b^q)/2\in\f_q$, $z=(b-b^q)/2\in\f_{q^2}\setminus\f_q$, and $k=z^2$. Then $k$ is a nonsquare in $\f_q^*$, and $z^q=-z$, $b=b_1+z$, $b^q=b_1-z$. The polynomials $A$ and $B$ in \eqref{2.5} and \eqref{2.6} become
\begin{equation}\label{3.2}
A(X)=(1+a+b_1-z)X^3-zX^2-kX+k((1-a+b_1)z-k),
\end{equation}
\begin{equation}\label{3.3}
B(X)=(1+a+b_1+z)X^3+zX^2-kX-k((1-a+b_1)z+k).
\end{equation}
The left side of \eqref{2.7} equals $z(C_3x^3+C_2(y)x^2+C_1(y)x+C_0(y))$, where
\begin{align}\label{3.4}
C_0(Y)&=k(-1+a^2+b_1-b_1^2+k)Y-ak^2,\\ \label{3.5}
C_1(Y)&=kY+(1+a+b_1)k,\\ \label{3.6}
C_2(Y)&=(1+a+b_1)Y+k,\\ \label{3.7}
C_3\hspace{1.5em}&=k-(1+a+b_1)^2\ne 0.
\end{align}
Hence \eqref{2.7} is equivalent to \eqref{2.8}. Let $E(Y)$ be given in \eqref{2.11}. We find that
\begin{equation}\label{3.8}
k^{-2}E(Y)=e_0+e_1Y+e_2Y^2+e_3Y^3+e_4Y^4,
\end{equation}
where
\begin{align}\label{3.9}
e_0=\,&k^2 (a^5- a^4 b_1- a^4+a^3 b_1^2- a^3 b_1- a^3 k+a^3+a^2 b_1^3+a^2 k+a^2\\
&- a b_1^4- a b_1^3- a b_1 k- a b_1+a k^2- a k- a+b_1^5- b_1^4- b_1^3 k\cr
&+b_1^3+b_1^2 k+b_1^2- b_1 k- b_1+1), \nonumber
\end{align}
\begin{equation} \label{3.10}
e_1= -k^2 (a^3+a^2 k+a k+b_1^3- b_1^2 k- b_1 k+k^2+1),
\end{equation}
\begin{equation} \label{3.11}
e_2= k C_3^2,
\end{equation}
\begin{equation} \label{3.12}
e_3= -k (-a^4+a^3- a^2 k- a b_1^3+a b_1 k- a k- a+b_1^3- b_1^2 k- b_1 k+k^2+1),
\end{equation}
\begin{equation} \label{3.13}
e_4= -(a+b_1+1)^2 (a^3+a^2 b_1+a^2- a b_1^2+a b_1+a k- a- b_1^3+b_1 k-1).
\end{equation}
Moreover,
\begin{equation}\label{3.14}
e_0e_3^2-e_1^2e_4=ak^4((1+a+b_1)^2-k)^3h_1
\end{equation}
and
\begin{equation}\label{3.15}
e_3^3-e_1e_4^2-e_2e_3e_4=k^2((1+a+b_1)^2-k)^3h_2,
\end{equation}
where $h_1$ and $h_2$ are polynomials in $a$, $b_1$ and $k$ which are given in Appendix \eqref{A1} and \eqref{A2}. By \eqref{2.17}, $h_1=h_2=0$. Treating $a$, $b_1$ and $k$ as variables, we find that
\begin{align}\label{3.16}
&\text{Res}(h_1,h_2;k)=-a^2(1+a+b_1)^{14}(1+a+a^2+ab_1)^3,\\ \label{3.17}
&\text{Res}(h_1,h_2;b_1)=a^3k^{14}(-1+a+a^2k)^3.
\end{align}
By \eqref{3.17},
\begin{equation}\label{3.18}
k=\frac{1-a}{a^2}.
\end{equation}
We then have
\begin{align*}
e_4\,&=\frac {-1}{a^2}(1+a+b_1)^2(1+a+a^2+ab_1)(a+b_1+ab_1+a^3-ab_1^2)\cr
&\kern17.5em\text{(by \eqref{3.13} and \eqref{3.18})}\cr
&=0\kern 20.4em\text{(by \eqref{3.16})}.
\end{align*}
Then by \eqref{2.16}, $e_3=0$ and $e_2$ is a square in $\f_q$. However, by \eqref{3.11}, $e_2$ is not a square in $\f_q$, which is a contradiction.

\subsection*{Case 1.2} Assume that $b\in\f_q$. Let $k\in\f_q^*$ be a nonsquare and let $z\in\f_{q^2}$ be such that $z^2=k$. Then $z^q=-z$. Now \eqref{2.5} and \eqref{2.6} become
\begin{align}\label{3.22}
A(X)\,&=(1+a+b)X^3-zX^2-kX+(1-a+b)kz,\\
\label{3.23}
B(X)\,&=(1+a+b)X^3+zX^2-kX-(1-a+b)kz,
\end{align}
and \eqref{2.7}, after multiplication by $z^{-1}$, becomes \eqref{2.8}, where
\begin{align}\label{3.24}
C_0(Y)\,&=-(1-a+b)kY,\\
\label{3.24a}
C_1\kern 1.7em &=k,\\
\label{3.25}
C_2(Y)\,&=Y,\\
\label{3.26}
C_3\kern 1.7em &=-(1+a+b)\ne 0.
\end{align}
Moreover,
\begin{equation}\label{3.27}
k^{-2}E(Y)=k^2(1+a+b)+kY^2+(1-a+b)Y^4,
\end{equation}
which is supposed to be a complete square. Therefore,
\begin{equation}\label{3.28}
k^2=k^2(1+a+b)(1-a+b),
\end{equation}
which gives $b=b^2-a^2$; this is the first condition in \eqref{1.4}. It remains to show that $(a^2-b^2)/a^2=-b/a^2$ is a square in $\f_q^*$, i.e., $-b$ is a square in $\f_q^*$. By \eqref{2.12}, the polynomial in \eqref{3.27} equals $D(Y)^2$ for some $D(Y)\in\f_q[Y]$. In particular, both $1+a+b$ and $1-a+b$ are squares in $\f_q$. By \eqref{3.28}, we may write $1+a+b=u^2$ and $1-a+b=u^{-2}$ for some $u\in\f_q^*$. Then
\[
-b=u^2-2+u^{-2}=(u-u^{-1})^2,
\]
which is a square in $\f_q^*$.

\section{Proof of Theorem~\ref{T3.1}, Case 2}

Assume that $a$ is not a square in $\f_{q^2}$ and $n$ is even. In this case,
$q\equiv 1\pmod 4$ and hence $\text{gcd}((q+1)/2,q-1)=1$. Let $a=\gamma^i$ and write $i=u(q+1)/2+v(q-1)$ for some $u,v\in\Bbb Z$. Then $f(\gamma^vX)=\gamma^vX(1+a'X^{q(q-1)}+b'X^{2(q-1)})$, where $a'=\gamma^{u(q+1)/2}$ and $b'=b\gamma^{2v(q-1)}$. Therefore we may assume that $a^2\in\f_q$ but $a\notin\f_q$. Hence $a^q=-a$.

\subsection*{Case 2.1} Assume that $b\notin\f_q$. Let $b_1=(b+b^q)/2\in\f_q$, $z=(b-b^q)/2\in\f_{q^2}\setminus\f_q$, and $k=z^2$. Then $k$ is a nonsquare in $\f_q^*$, and $z^q=-z$, $b=b_1+z$, $b^q=b_1-z$. Since both $a^2$ and $z^2$ are nonsquares in $\f_q^*$, $(a/z)^2$ is a square in $\f_q^*$, whence $a/z\in\f_q^*$. Write $a=a_1z$, where $a_1\in\f_q^*$. Now \eqref{2.5} and \eqref{2.6} become
\begin{align}\label{3.29}
A(X)\,&=(1+b_1-z-a_1z)X^3-zX^2-kX+k(-k+a_1k+z+b_1z),\\
\label{3.30}
B(X)\,&=(1+b_1+z+a_1z)X^3+zX^2-kX+k(-k+a_1k-z-b_1z),
\end{align}
and \eqref{2.7}, after multiplication by $z^{-1}$, becomes \eqref{2.8}, where
\begin{align}\label{3.31}
C_0(Y)\,&=k(-1+b_1-b_1^2+k-a_1^2k)Y+k^2a_1(1+b_1),\\
\label{3.32}
C_1(Y)\,&=k(1+a_1)Y+k(1+b_1),\\
\label{3.33}
C_2(Y)\,&=(1+b_1)Y+k(1+a_1),\\
\label{3.34}
C_3\kern 1.7em &=-(1+b_1)^2+(1+a_1)^2k.
\end{align}
We claim that $C_2(Y)\ne 0$ and $C_3\ne 0$, i.e., $(a_1,b_1)\ne (-1,-1)$. Otherwise, $C_0(Y)$, $C_1(Y)$, $C_2(Y)$, $C_3$ are all 0, which is impossible since \eqref{2.8} has a unique solution $x\in\f_q$ for any given $y\in\f_q$. We have
\begin{equation}\label{3.35}
k^{-2}E(Y)=e_0+e_1Y+e_2Y^2+e_3Y^3+e_4Y^4,
\end{equation}
where
\begin{align}\label{3.36}
e_0=\,&-(b_1+1) k^2 (a_1^4 k^2+a_1^2 b_1^2 k+a_1^2 b_1 k- a_1 b_1^2 k- a_1 b_1 k\\
&+a_1 k^2- b_1^4- b_1^3+b_1^2 k+b_1 k- b_1-1), \nonumber
\end{align}
\begin{align}\label{3.37}
e_1=\,& (a_1+1) k^2 (a_1^4 k^2- a_1^3 k^2+a_1^2 b_1^2 k+a_1^2 b_1 k- a_1 b_1^2 k\\
&- a_1 b_1 k+a_1 k^2- b_1^3+b_1^2 k+b_1 k- k^2-1),\nonumber
\end{align}
\begin{equation}\label{3.38}
e_2= k C_3^2,
\end{equation}
\begin{align}\label{3.39}
e_3=\,& -k (a_1^5 k^2- a_1^4 k^2- a_1^3 b_1^2 k- a_1^3 b_1 k+a_1^3 k^2+a_1^2 k^2+a_1 b_1^4\\
&- a_1 b_1^3+a_1 b_1- a_1 k^2- a_1+b_1^3- b_1^2 k- b_1 k+k^2+1),\nonumber
\end{align}
\begin{equation}\label{3.40}
e_4= (b_1+1)^2 (a_1^2 b_1 k- a_1^2 k- a_1 k+b_1^3- b_1 k+1).
\end{equation}
Moreover,
\begin{equation}\label{3.41}
e_0e_3^2-e_1^2e_4=k^4a_1(1+b_1)((1+b_1)^2-k(1+a_1)^2)^3h_1,
\end{equation}
\begin{equation}\label{3.42}
e_3^3-e_1e_4^2-e_2e_3e_4=-k^2((1+b_1)^2-k(1+a_1)^2)^3h_2,
\end{equation}
where $h_1$ and $h_2$ are given in Appendix \eqref{A3} and \eqref{A4}.

Recall that $(a_1,b_1)\ne (-1,-1)$. We claim that $a_1\ne -1$ and $b_1\ne -1$. In fact, if $a_1=-1$, then $h_1=-(1+b_1)^6$, whence $b_1=-1$; if $b_1=-1$, then $h_2=-k^4(1+a_1)^9$, whence $a_1=-1$. Now it follows from \eqref{3.41} and \eqref{3.42} that $h_1=h_2=0$. We find that
\begin{align}\label{3.44}
\text{Res}(h_1,h_2;k)\,&=a_1^3(1+a_1)^{18}(1+b_1)^{17}(1-a_1+a_1b_1)(1+a_1^2-a_1b_1+a_1b_1^2)^3,\\
\label{3.45}
\text{Res}(h_1,h_2;b_1)\,&=k^{17}a_1^5(1+a_1)^{33}(-1+a_1^2k)(1+a_1-a_1k+a_1^2k+a_1^2k^2+a_1^3k^2)^3.
\end{align}
By \eqref{3.45},
\begin{equation}\label{3.45.0}
0=1+a_1-a_1k+a_1^2k+a_1^2k^2+a_1^3k^2=(1+a_1k)^2+a_1(1-a_1k)^2,
\end{equation}
so
\begin{equation}\label{3.46}
a_1=-\Bigl(\frac{1+a_1k}{1-a_1k}\Bigr)^2.
\end{equation}
Assume that in \eqref{3.44}, $1+a_1^2-a_1b_1+a_1b_1^2= 0$, i.e., $(1+a_1)^2+a_1(1+b_1)^2=0$, i.e.,
\begin{equation}\label{3.47}
a_1=-\Bigl(\frac{1+a_1}{1+b_1}\Bigr)^2.
\end{equation}
Then
\begin{equation}\label{3.48}
\frac{1+a_1k}{1-a_1k}=\pm \frac{1+a_1}{1+b_1}.
\end{equation}

First assume that the ``$+$'' sign holds in the above. Then
\begin{equation}\label{3.49}
k=\frac{a_1-b_1}{a_1(a_1+b_1-1)}.
\end{equation}
(Note that $a_1+b_1-1\ne0$ since otherwise we also have $a_1-b_1=0$ and hence $(a_1,b_1)=(-1,-1)$.) Using \eqref{3.49} in \eqref{3.40} gives
\[
e_4=\frac{(1+b_1)^2(-a_1+a_1b_1+b_1^2)(1+a_1^2-a_1b_1+a_1b_1^2)}{a_1(a_1+b_1-1)}=0.
\]
It follows, as we saw before (at the end of Section 4, Case 1.1), that $e_2$ is a square in $\f_q$, which contradicts \eqref{3.38}.

Therefore the ``$-$'' sign holds in \eqref{3.48}. Then
\begin{equation}\label{3.51}
k=\frac{a_1+b_1-1}{a_1(a_1-b_1)}.
\end{equation}
Using \eqref{3.51} in \eqref{A3} gives
\begin{equation}\label{3.52}
h_1=\frac T{a_1^3(a_1-b_1)^3},
\end{equation}
where $T$ is given in Appendix \eqref{A6}. Moreover,
\[
\text{Res}(T,1+a_1^2-a_1b_1+a_1b_1^2;a_1)=(1+b_1)^{18}.
\]
Thus $b_1=-1$, which is a contradiction.

Therefore, we have proved that in \eqref{3.44}, $1+a_1^2-a_1b_1+a_1b_1^2\ne 0$. Thus $1-a_1+a_1b_1=0$, i.e., $b_1=(a_1-1)/a_1$. Using this substitution in \eqref{A3} and \eqref{A4} gives
\begin{equation}\label{3.53}
h_1=\frac{(1+a_1)^5}{a_1^5}S_1,
\end{equation}
\begin{equation}\label{3.54}
h_2=-\frac{(1+a_1)^7}{a_1^7}S_2,
\end{equation}
where
\begin{equation}\label{3.55}
S_1= a_1^7 k^3- a_1^6 k^3+a_1^5 k^3+a_1^5 k^2+a_1^4 k^2+a_1^3 k^2- a_1^3 k+a_1^2 k+a_1 k- a_1-1,
\end{equation}
\begin{equation}\label{3.56}
S_2= a_1^9 k^4- a_1^8 k^4+a_1^7 k^4- a_1^4 k^2- a_1 k- a_1-1.
\end{equation}
Let $L=(1+a_1k)^2+a_1(1-a_1k)^2$, which is the polynomial in \eqref{3.45.0}. We find useful information by computing the following resultants:
\begin{align*}
\text{Res}(S_1,L;a_1)\,&=-k^8(-1+k+k^2+k^3)(1-k^2+k^3+k^5),\cr
\text{Res}(S_2,L;a_1)\,&=-k^{10}(-1-k+k^3)(-1+k+k^2+k^3)(-1-k-k^2+k^3).
\end{align*}
The factors in the above are irreducible polynomials in $k$ over $\f_3$. Hence we must have $-1+k+k^2+k^3=0$. Then $k\in\f_{3^3}$. Since $[\f_q:\f_3]$ is even, $k$ is a square in $\f_q$, which is a contradiction.

\subsection*{Case 2.2} Assume that $b\in\f_q$. Let $z=a$ and $k=a^2$. Now \eqref{2.5} and \eqref{2.6} become
\begin{align}\label{3.57}
A(X)\,&=(1-a+b)X^3-aX^2-kX+k(a+ab+k),\\
\label{3.58}
B(X)\,&=(1+a+b)X^3+aX^2-kX-k(a+ab-k),
\end{align}
and \eqref{2.7}, after multiplication by $a^{-1}$, becomes \eqref{2.8}, where
\begin{align}\label{3.60}
C_0(Y)&=k(-1+b-b^2-k)Y+k^2(1+b),\\ \label{3.61}
C_1(Y)&=kY+k(1+b),\\ \label{3.62}
C_2(Y)&=(1+b)Y+k,\\ \label{3.63}
C_3\hspace{1.5em}&=-(1+b)^2+k\ne 0.
\end{align}
Moreover,
\begin{equation}\label{3.64}
k^{-2}E(Y)=e_0+e_1Y+e_2Y^2+e_3Y^3+e_4Y^4,
\end{equation}
where
\begin{align}\label{3.65}
e_0\,&=-k^2(1+b)(-1-b-b^3-b^4+bk+b^2k+k^2),\\
\label{3.66}
e_1\,&=k^2(-1-b^3+bk+b^2k+k^2),\\
\label{3.67}
e_2\,&=kC_3^2,\\
\label{3.68}
e_3\,&=-k(-1+b-b^3+b^4-bk-b^2k+k^2),\\
\label{3.69}
e_4\,&=(1+b)^2(1+b^3-k+bk).
\end{align}
We have
\begin{equation}\label{3.70}
e_0e_3^2-e_1^2e_4=k^4(1+b)((1+b)^2+k)^2((1+b)^2-k)^3(b+b^2+k).
\end{equation}
In \eqref{3.70}, $1+b\ne 0$. (Otherwise, $e_4=0$, and as we saw before, this implies that $e_2$ is a square in $\f_2$, which is impossible.) Since $-1$ is a square in $\f_q$, we also have $(1+b)^2+k\ne 0$. Therefore by \eqref{3.70}, $b+b^2+k=0$, i.e., $b+b^2+a^2=0$; this is the first condition in \eqref{1.4}. Since $k=-b-b^2$, \eqref{3.69} gives $e_4=(1+b)^3$. Since $e_4$ is a square in $\f_q$, $1+b$ is a square in $\f_q^*$. Now we have
\[
1-\Bigl(\frac ba\Bigr)^{q+1}=1+\frac{b^2}{a^2}=\frac{a^2+b^2}{a^2}=\frac{-b}{-b(1+b)}=\frac 1{1+b},
\]
which is a square in $\f_q^*$.

\section{Proof of Theorem~\ref{T3.1}, Case 3}

Assume that $a$ is not a square in $\f_{q^2}$ and $n$ is odd. Let $z\in\f_{q^2}$ be such that $z^2=-1$. Then $o(1+z)=8$. Since $q^2-1\equiv 8\pmod{16}$, we may write $a=u(1+z)$, where $u$ is a square in $\f_{q^2}^*$. By the first paragraph of Section~4, we may further assume that $u\in\f_q^*$. Write $b=v+wz$, where $v,w\in\f_q$. We have $z^q=-z$, $a^q=u(1-z)$, and $b^q=v-wz$. Then \eqref{2.5} and \eqref{2.6} become
\begin{align}\label{3.74}
A(X)\,&=(1+u+v-uz-wz)X^3-zX^2+X+u-w-z(1-u+v),\\
\label{3.75}
B(X)\,&=(1+u+v+uz+wz)X^3+zX^2+X+u-w+z(1-u+v),
\end{align}
and \eqref{2.7}, after multiplication by $z^{-1}$, becomes \eqref{2.8}, where
\begin{align}\label{3.76}
C_0(Y)&=(1+u^2-v+v^2+w^2)Y+u(1+v-w),\\ \label{3.77}
C_1(Y)&=-(u+w)Y-(1+u+v),\\ \label{3.78}
C_2(Y)&=(1+u+v)Y-(u+w),\\ \label{3.79}
C_3\hspace{1.5em}&=-(1+u+v)^2-(u+w)^2.
\end{align}
We claim that $C_3\ne 0$ and $C_2(Y)\ne 0$. Otherwise, $1+u+v=u+w=0$, and hence $C_1(Y)=0$. Then \eqref{2.8} cannot have a unique solution for $x$, which is a contradiction.

We now compute $E(Y)$ in \eqref{2.11} with $k=-1$:
\begin{equation}\label{3.80}
E(Y)=e_0+e_1Y+e_2Y^2+e_3Y^3+e_4Y^4,
\end{equation}
where
\begin{align}\label{3.81}
e_0=\,& -(-1 + u - u^2 + u^3 + u^5 + v + u v - u^2 v - u^4 v - v^2 +
   u^2 v^2 - u^3 v^2 \\
   &- v^3 + u v^3 + u^2 v^3 + v^4 + u v^4 - v^5 + u^2 w - u^3 w + u v w + u^2 v w \cr
   &- u v^2 w + u v^3 w - u w^2 + u^2 w^2 - u^3 w^2 - v w^2 - u v w^2 + v^2 w^2 \cr
   &- v^3 w^2 + u w^3 + u v w^3 - u w^4), \nonumber
\end{align}
\begin{align}\label{3.82}
e_1=\,& -(u + w) (1 + u^4 + u^2 v + u^2 v^2 + v^3 + u^2 w - u^3 w - u v w - u v^2 w \\
   &- u w^2 - u^2 w^2 + v w^2 + v^2 w^2 - u w^3 + w^4),  \nonumber
\end{align}
\begin{equation}\label{3.83}
e_2=-C_3^2,
\end{equation}
\begin{align}\label{3.84}
e_3=\,& -(u + u^5 - u v - u^3 v - u^3 v^2 + u v^3 - u v^4 - w +
   u w - u^3 w \\
   &- u^4 w - v^3 w + u v^3 w - u^3 w^2 - u w^3 +
    u^2 w^3 - v w^3 + u v w^3 \cr
    &- v^2 w^3 + u w^4 - w^5),  \nonumber
\end{align}
\begin{equation}\label{3.85}
e_4=(1 + u + v)^2 (1 + u + u^3 - u v + u^2 v + u v^2 + v^3 + u w +
   u w^2 + v w^2).
\end{equation}
We have $e_4\ne 0$, in particular, $1+u+v\ne 0$. (Otherwise, as we saw before, $e_2$ is a square in $\f_q$, which is impossible.) We find that
\begin{equation}\label{3.86}
e_0e_3^2-e_1^2e_4=uC_3^3h_1,
\end{equation}
\begin{equation}\label{3.87}
e_3^3-e_1e_4^2-e_2e_3e_4=-C_3^3h_2,
\end{equation}
where $h_1$ and $h_2$ are given in Appendix \eqref{A7} and \eqref{A8}. Moreover,
\begin{align}\label{3.88}
\text{Res}(h_1,h_2;u)=\,&v^6w(1+v-w)^{42}K_1(v,w)^3K_2(v,w),\\
\label{3.89}
\text{Res}(h_1,h_2;v)=\,&-u^3(1+u)^2(1+u^2)(u+w)^{42}\\
&\cdot(-1-u+u^2+w+w^2)^2(u^2+w+w^2)L(u,w)^3,\cr
\label{3.90}
\text{Res}(h_1,h_2;w)=\,& -u^5(1+u)^2(1+u^2)v^6(1+u+v)^{42}M(u,v)^3,
\end{align}
where
\begin{equation}\label{3.91}
K_1= v^5- v^4 w+v^4+v^3 w^2- v^3- v^2 w^3+v^2 w^2- v^2- v w^2+v w- v+w-1,
\end{equation}
\begin{align}\label{3.92}
K_2=\,& v^7+v^6 w^2+v^6+v^5 w^2- v^5 w- v^4 w^3- v^4 w^2+v^4 w- v^4\\
&- v^3 w^4- v^3 w^3- v^3 w^2+v^2 w^5- v^2 w^4+v^2 w^3- v^2 w^2- v^2 w\cr
&- v w^6- v w^4+v w+v+w^8- w^7+w^6- w^3+w^2-1, \nonumber
\end{align}
\begin{equation} \label{3.93}
L= u^5+u^4 w+u^3 w^2+u^3 w+u^2 w^3+u^2 w^2+u^2 w- u^2- u w^2+u w- u- w,
\end{equation}
\begin{equation} \label{3.94}
M= u^3 v^2+u^2 v^3+u^2 v^2+u^2 v+u v^2- u v- u- v-1.
\end{equation}

\subsection*{Case 3.1} Assume that $v\ne0$.
We claim that $1+u\ne 0$. If, to the contrary, $u=-1$, then 
\begin{align}\label{3.95}
h_1\,&=(v^2+(w-1)^2)^2H_1,\\
\label{3.96}
h_2\,&=(v^2+(w-1)^2)H_2,
\end{align}
where
\begin{equation}\label{3.97}
H_1=v^3- v^2 w+v^2+v w^3- v w^2- v w+v-w^4+w^3+w-1,
\end{equation}
\begin{align}\label{3.98}
H_2=\,& v^5 w+v^5+v^4 w^3- v^4+v^3 w^3- v^3 w^2- v^3 w+v^3- v^2 w^5- v^2 w^4\\
&+v^2 w^3+v^2 w^2+v^2 w- v^2+w^7- w^6+w^4- w^3+w-1. \nonumber
\end{align}
Moreover,
\begin{align}\label{3.99}
\text{Res}(H_1,H_2;v)\,&=(1+w)^6(2+w)^{19},\\
\label{3.100}
\text{Res}(H_1,H_2;w)\,&=v^{19}(1+v)^6.
\end{align}
By \eqref{3.99}, $w=\pm1$. If $w=1$, then $h_1=v^7\ne 0$, which is a contradiction. If $w=-1$, then $h_1=(1+v)(1+v^2)^2(2+v+v^2)$, whence $v=-1$. However, $e_4|_{(u,v,w)=(-1,-1,-1)}=0$, which is a contradiction. Therefore, we have proved that $1+u\ne 0$. It follows from \eqref{3.90} that
\begin{equation}\label{3.101}
M=0.
\end{equation}

In \eqref{3.89}, note that $-1-u+u^2+w+w^2=(1+u)^2+(1-w)^2\ne 0$. We further claim that $u+w\ne 0$. If, to the contrary, $w=-u$, then $h_1=-u(1+u+v)^7\ne 0$, which is a contradiction. Now \eqref{3.89} gives
\begin{equation}\label{3.102}
(u^2+w+w^2)L=0.
\end{equation}

In \eqref{3.88}, we claim that $w\ne 0$. Otherwise,
\begin{equation}\label{3.103}
\text{Res}(h_1,h_2;v)=-u^{50}(1+u)^2(1+u^2)(-1-u+u^2)^2(-1-u+u^4)^3\ne 0,
\end{equation}
which is a contradiction. We further claim that $1+v-w\ne 0$. Otherwise,
\begin{align}\label{3.104}
h_1\,&=(1+u+v)^6(-1+u^2-v),\\
\label{3.105}
h_2\,&=-(1+u+v)^6v(1+u-v+v^2).
\end{align}
However,
\[
\text{Res}(-1+u^2-v,1+u-v+v^2;v)=u(1+u)^3\ne 0,
\]
which is a contradiction. Now \eqref{3.88} gives
\begin{equation}\label{3.106}
K_1K_2=0.
\end{equation}

To recap, we have
\begin{equation}\label{3.107}
K_1K_2=0,\quad (u^2+w+w^2)L=0,\quad M=0.
\end{equation}

First assume that $u^2+w+w^2=0$. We find that
\begin{align}\label{3.108}
S:=\,&\text{Res}(u^2+w+w^2,M;u)\\
=\,& v^6 w^4- v^6 w^3+v^6 w^2- v^5 w^4+v^5 w^3- v^5 w^2+v^4 w^6+v^4 w^4+v^4 w^2\cr
&+v^3 w^4-v^3 w^3- v^3 w+v^2+v w^2+v w- v+w^2+w+1. \nonumber
\end{align}

If $K_1=0$, we compute the following resultants:
\begin{align}\label{3.109}
\text{Res}(S,K_1;v)\,&=w(1+w)(-1+w)(1+w^2-w^3+w^4)P_{21}(w),\\
\label{3.110}
\text{Res}(S,K_1;w)\,&=v^{13}(1+v)^2(-1+v+v^4)Q_{21}(v),
\end{align}
where
\begin{align}\label{3.111}
P_{21}(X)=\,& X^{21}- X^{19}- X^{18}- X^{17}+X^{16}+X^{14}- X^{13}+X^{11}\\
&- X^{10}- X^7+X^6- X^3+X^2-1, \nonumber
\end{align}
\begin{align}\label{3.112}
Q_{21}(X)=\,& X^{21}+X^{20}- X^{18}- X^{17}- X^{16}+X^{15}+X^{12}- X^{11}\\
&+X^{10}- X^9- X^8- X^6+X^5+X^4- X^3+X-1, \nonumber
\end{align}
and all factors in \eqref{3.109} and \eqref{3.110} are irreducible polynomials over $\f_3$. We claim that in \eqref{3.109}, $(1+w)(-1+w)\ne 0$. Otherwise, $w=\pm 1$ and $u^2=-w-w^2=0$ or $1$. Since $u\ne 0,-1$, we must have $u=1$ and $w=1$. However,
\[
h_1|_{(u,w)=(1,1)}=-v(1-v+v^4-v^5+v^6)\ne 0,
\]
which is a contradiction. Moreover, we claim that in \eqref{3.110}, $1+v\ne 0$. Otherwise, $M=u(1+u)^2\ne 0$. Therefore, \eqref{3.109} and \eqref{3.110} give that $P_{21}(w)=0$ and $Q_{21}(v)=0$. 
Since $u^2=-w-w^2$, we have $u\in\f_q\cap\f_{3^{2\cdot 21}}=\f_{3^{21}}$. It follows that $(-w-w^2)^{(3^{21}-1)/2}=1$ and hence
\begin{equation}\label{u^2}
u^2=-w-w^2=(-w-w^2)^{(3^{21}+1)/2}.
\end{equation}
Therefore $u=\pm u_0$, where
\begin{align}
u_0=(-w-w^2)^{(3^{21}+1)/4}=\,&w^{20}+w^{19}- w^{18}- w^{17}- w^{15}- w^{14}+w^{12}- w^{11}\\
&+w^{10}+w^9+w^8+w^7-w^6+w^5- w-1.\nonumber
\end{align}
We find that
\begin{equation}\label{res-u0}
\text{Res}(h_1|_{u=u_0},h_2|_{u=u_0},w)\quad \text{and}\quad \text{Res}(h_1|_{u=-u_0},h_2|_{u=-u_0},w)
\end{equation}
are polynomials over $\f_3$ in $v$ which are not divisible by $Q_{21}(v)$; see \eqref{app-res-u0} and \eqref{app-res-u0-}. Thus the resultants in \eqref{res-u0} are nonzero, which is a contradiction.

If $K_2=0$, we compute similar resultants:
\begin{align}\label{3.113}
\text{Res}(S,K_2;v)\,&=-w^3(-1+w)^2\Pi_1(w)P_{15}(w),\\
\label{3.114}
\text{Res}(S,K_2;w)\,&=-v^2(1+v)^3\Pi_2(v)Q_9(v)^2Q_{15}(v),
\end{align}
where $\Pi_1$ and $\Pi_2$ (given in Appendix \eqref{Pi1} and \eqref{Pi2}) are products of irreducible polynomials of even degree over $\f_3$, and
\begin{equation}\label{3.115}
P_{15}(X)=X^{15}- X^{13}- X^{11}+X^{10}- X^9- X^7+X^6- X^3- X^2+X+1,
\end{equation}
\begin{equation}\label{3.116}
Q_9(X)=X^9+X^8+X^7+X^6+X^4- X^3- X^2- X-1,
\end{equation}
\begin{equation}\label{3.117}
Q_{15}(X)=X^{15}+X^{14}- X^{13}- X^{11}+X^9- X^7- X^6+X^3-1,
\end{equation}
which are irreducible over $\f_3$. It follows that $P_{15}(w)=0$ and $Q_9(v)Q_{15}(v)=0$. 
Same as \eqref{u^2}, we have
\[
u^2=(-w-w^2)^{(3^{15}+1)/2}.
\]
Thus $u=\pm u_1$, where
\[
u_1=(-w-w^2)^{(3^{15}+1)/4}= -w^{14}+w^{12}- w^{11}+w^{10}+w^7+w^6+w^5- w^4- w^3- w^2+w-1.
\]
We find that
\begin{equation}\label{res-u1}
\text{Res}(h_1|_{u=u_1},h_2|_{u=u_1};w)\quad \text{and}\quad \text{Res}(h_1|_{u=-u_1},h_2|_{u=-u_1};w)
\end{equation}
are polynomials over $\f_3$ in $v$ which are divisible by neither $Q_9(v)$ nor $Q_{15}(v)$; see \eqref{app-res-u1} and \eqref{app-res-u1-}. Thus the resultants in \eqref{res-u1} are nonzero, which is a contradiction.

Now assume that $u^2+w+w^2\ne 0$. Then by \eqref{3.107}, $L=0$. We find that
\begin{equation}\label{3.118}
\text{Res}(L,M;u)=K_1Q,
\end{equation}
where
\begin{align}\label{3.119}
Q=\,& v^{10} w^2- v^9 w^3- v^9 w^2- v^9 w- v^8 w^4+v^8 w^3- v^8- v^7 w^5\\
&- v^7 w^3+v^7
   w^2+v^7 w+v^6 w^6+v^6 w^4+v^6 w^2- v^6 w- v^6- v^5 w^6\cr
   &+v^5 w^4- v^5 w^3+v^4 w^6-
   v^4 w^4- v^4 w^3- v^4 w- v^4+v^3 w^5\cr
   &+v^3 w^4+v^3 w^3- v^2 w^4- v^2 w^2+v^2- v
   w^3+w^2- w+1. \nonumber
\end{align}
If $K_1\ne 0$, then $K_2=0$ and $Q=0$. We compute the following resultants:
\begin{align}\label{3.120}
\text{Res}(Q,K_2;v)\,&=\Gamma_1(w)P_7(w)P_{21}(w),\\
\label{3.121}
\text{Res}(Q,K_2;w)\,&=\Gamma_2(v)Q_7(v)Q_{21}(v),
\end{align}
where $\Gamma_1$ and $\Gamma_2$ (given in Appendix \eqref{Gamma1} and \eqref{Gamma2}) are products of irreducible polynomials of even degree over $\f_3$ and
\begin{equation}\label{3.122}
P_7(X)=X^7- X^5-1,
\end{equation}
\begin{align}\label{3.123}
P_{21}(X)=\,&X^{21}+X^{20}+X^{19}+X^{17}+X^{16}- X^{15}- X^{13}- X^{12}\\
&+X^{11}- X^{10}- X^9- X^8- X^6- X^5- X^2-1, \nonumber
\end{align}
\begin{equation}\label{3.124}
Q_7(X)=X^7- X^6+X^3- X^2+X+1,
\end{equation}
\begin{align}\label{3.125}
Q_{21}(X)=\,&X^{21}+X^{20}+X^{19}+X^{18}+X^{16}- X^{15}+X^{14}- X^{13}+X^{11}\\
&- X^{10}- X^9- X^8+X^6- X^5- X^2+X-1, \nonumber
\end{align}
which are irreducible over $\f_3$. It follows that $P_7(w)P_{21}(w)=0$ and $Q_7(v)Q_{21}(v)=0$. 
We claim that either $P_7(w)=Q_7(v)=0$ or $P_{21}(w)=Q_{21}(v)=0$. In fact,
\begin{align*}
&\text{Res}(P_7(w),K_2;w)\cr
=\,&Q_7(v)(v^7+v^6+v^5- v^4+v^2- v-1)\cr
&\cdot (v^7- v^6+v^4- v^3- v^2-1)
   (v^{14}- v^8- v^6- v^2+1)\cr
   &\cdot (v^{14}+v^{13}+v^{12}- v^{11}- v^{10}-
   v^9- v^8- v^7+v^6+v^5+v^4- v^2- v-1),
\end{align*}
which is not divisible by $Q_{21}(v)$. Hence $P_7(w)=0$ implies $Q_7(v)=0$. Similarly,
\begin{align*}
&\text{Res}(Q_7(v),K_2;v)\cr
=\,&P_7(w) (w^7- w^5+1) (w^7+w^6+w^5- w^4+w^3+w^2+1)\cr
&\cdot (w^7-
   w^6+w^5+w^4+w^3- w^2-1)\cr
   &\cdot (w^{14}+w^{13}+w^{12}- w^{11}+w^{10}+w^8+w^7-
   w^3+w-1)\cr
   &\cdot (w^{14}+w^{13}- w^{12}- w^{11}+w^{10}- w^9- w^5+w^4-
   w^3+w^2+1),
\end{align*}
which is not divisible by $P_{21}(w)$. Hence $Q_7(v)=0$ implies $P_7(w)=0$. Therefore the claim is proved.

First assume that $P_7(w)=Q_7(v)=0$. From $Q_7(v)=0$, we have $v^{-2}=v^6+v^5+v^4+v^2+v-1$ and
\begin{equation}\label{v-2M}
v^{-2}M=u^3+u^2 (-v^6+v^5- v^2- v)+u (v^5- v^4+v)+v^5- v^4+v-1,
\end{equation}
which is a monic cubic in $u$. Using the relations $v^{-2}M=0$, $P_7(w)=0$ and $Q_7(v)=0$ to reduce $h_1$ and $h_2$, we get
\begin{equation}\label{redp7q7}
h_i=h_{i0}+h_{i1}u+h_{i2}u^2,\quad i=1,2,
\end{equation}
where $h_{ij}=h_{ij}(v,w)$ are given in \eqref{app-hij-1} -- \eqref{app-hij-6}. We have
\begin{equation}\label{A0A1}
h_{22}h_1-h_{12}h_2\equiv A_0+A_1u\pmod{(P_7(w),Q_7(v))},
\end{equation}
where $A_i=A_i(v,w)$ are given in \eqref{app-A0} and \eqref{app-A1}. The resultant $\text{Res}(K_2,A_1;v)$ is a polynomial in $w$ such that
\[
\text{Res}(K_2,A_1;v)\equiv w^6+w^2- w \not\equiv 0\pmod{P_7(w)},
\]
so $A_1\ne 0$. By \eqref{A0A1}, $u=-A_0/A_1$. We now compute $h_1|_{u=-A_0/A_1}$ from \eqref{redp7q7}; the result, after reduction modulo $P_7(w)$ and $Q_7(v)$, is
\begin{equation}\label{U/V}
h_1|_{u=-A_0/A_1}=\frac UV,
\end{equation}
where $U=U(v,w)$ and $V=V(v,w)$ are given in \eqref{app-U} and \eqref{app-V}. Thus $U=0$. However, $\text{Res}(K_2,U;v)$ is a polynomial in $w$ such that
\[
\text{Res}(K_2,U;v)\equiv w^6+w^5+w^2 \not\equiv 0\pmod{P_7(w)}.
\]
Hence $\text{Res}(K_2,U;v)\ne 0$, which is a contradiction.

Next, Assume that $P_{21}(w)=Q_{21}(v)=0$. We reach a contradiction by similar but lengthier computations. From $Q_{21}(v)=0$, we have
\[
v^{-2}= v^{20}- v^{19}- v^{18}- v^{17}+v^{16}+v^{15}- v^{11}+v^{10}+v^8+v^7- v^6+v^5- v^3-
   v
\]
and
\begin{align*}
v^{-2}M=\,& u^3+u^2 (v^{20}+v^{19}+v^{18}+v^{17}+v^{15}- v^{14}+v^{13}- v^{12}+v^{10}- v^9- v^8-
   v^7\cr
   &+v^5- v^4-1)+u (v^{20}- v^{16}+v^{15}+v^{14}-
   v^{13}+v^{12}+v^{11}+v^{10}+v^9\cr
   &+v^6+v^5+v^4+v^3- v)+v^{20}- v^{16}+v^{15}+v^{14}-
   v^{13}+v^{12}+v^{11}+v^{10}\cr
   &+v^9+v^6+v^5+v^4+v^3- v-1,
\end{align*}
which is a monic cubic in $u$. Reducing $h_1$ and $h_2$ using the relations $v^{-2}M=0$, $P_{21}(w)=0$ and $Q_{21}(v)=0$ gives
\begin{equation}\label{redp21q21}
h_i=h_{i0}+h_{i1}u+h_{i2}u^2,\quad i=1,2,
\end{equation}
where $h_{ij}=h_{ij}(v,w)$ are given in \eqref{app-Hij-1} -- \eqref{app-Hij-6}. Moreover,
\begin{equation}\label{A0A1-1}
h_{22}h_1-h_{12}h_2\equiv A_0+A_1u\pmod{(P_{21}(w),Q_{21}(v))},
\end{equation}
where $A_i=A_i(v,w)$ are given in \eqref{app-A0-1} and \eqref{app-A1-1}. The resultant $\text{Res}(K_2,A_1;v)$ is a polynomial in $w$ such that
\begin{align*}
\text{Res}(K_2,A_1;v)\equiv\,& -w^{20}+w^{19}+w^{18}+w^{15}+w^{14}+w^{13}+w^{12}+w^{11}\cr
&+w^{10}+w^7- w^6+w^5- w^4+w^3- w^2+1\cr
 \not\equiv\,& 0\pmod{P_{21}(w)}.
\end{align*}
Thus $A_1\ne 0$ and hence $u=-A_0/A_1$ by \eqref{A0A1-1}.  We compute $h_1|_{u=-A_0/A_1}$ from \eqref{redp21q21}; the result, after reduction modulo $P_{21}(w)$ and $Q_{21}(v)$, is
\begin{equation}\label{U/V-1}
h_1|_{u=-A_0/A_1}=\frac UV,
\end{equation}
where $U=U(v,w)$ and $V=V(v,w)$ are given in \eqref{app-U-1} and \eqref{app-V-1}. Thus $U=0$. However, $\text{Res}(K_2,U;v)$ is a polynomial in $w$ such that
\begin{align*}
\text{Res}(K_2,U;v)\,&\equiv w^{20}+w^{16}- w^{15}+w^{14}- w^{13}+w^{12}- w^{10}- w^9- w^6+w^3+w^2+w\cr
& \not\equiv 0\pmod{P_{21}(w)}.
\end{align*}
Hence $\text{Res}(K_2,U;v)\ne 0$, which is a contradiction.

Now assume that $K_1=0$. In this case we solve the system
\[
\begin{cases}
K_1(v,w)=0,\cr
L(u,w)=0
\end{cases}
\]
for $w$ in terms of $u$ and $v$. We find that
\begin{equation}\label{3.126}
\alpha K_1+\beta L=\lambda w-\eta,
\end{equation}
where
\begin{align}\label{3.127}
\alpha=\,& -u^2 (u^3 v^3+u^3 v^2 w- u^3 v^2- u^3 v+u^2 v^4+u^2 v^3 w+u^2 v^3- u^2 v^2 w\\
&- u^2 v^2- u^2 v w+u^2 v- u^2- u v^3- u v^2 w- u v^2+u v- v^2), \nonumber
\end{align}
\begin{align}\label{3.128}
\beta=\,& v^2 (u^4 v^2+u^3 v^3- u^3 v^2 w- u^3 v^2- u^3 v- u^2 v^3 w+u^2 v^2 w\\
&+u^2 v^2+u^2 v w+u^2 v- u^2- u v^3+u v^2 w+u v- v^2), \nonumber
\end{align}
\begin{align}\label{3.129}
\lambda=\,& u^7 v^4+u^6 v^5- u^6 v^4- u^6 v^2+u^5 v^6- u^5 v^4- u^5 v^3- u^5 v^2+u^5 v\\
&+u^4 v^7- u^4 v^6- u^4 v^5- u^4 v^3- u^4 v+u^4- u^3 v^5- u^3 v^4+u^3 v^2\cr
&- u^3 v- u^2 v^6- u^2 v^5+u^2 v^3- u^2 v^2+u v^5- u v^4- u v^3+v^4, \nonumber
\end{align}
\begin{align}\label{3.130}
\eta=\,& -u (u^8 v^4+u^7 v^5- u^7 v^4- u^7 v^3+u^6 v^4+u^6 v^3- u^6 v^2- u^5 v^5\\
&- u^5 v^4+u^5 v^3- u^4 v^8+u^4 v^4+u^4 v^2- u^4 v- u^3 v^9+u^3 v^8\cr
&+u^3 v^7- u^3 v^6- u^3+u^2 v^8- u^2 v^7- u^2 v^6+u^2 v^5+u^2 v^4\cr
&+u^2 v^2+u^2 v+u v^7+u v^6+u v^3- u v^2+v^4). \nonumber
\end{align}
If $\lambda\ne 0$, we have $w=\eta/\lambda$, and more importantly, by \eqref{3.85},
\begin{equation}\label{3.131}
e_4|_{w=\eta/\lambda}=\lambda^{-2}(1+u+v)^2MN,
\end{equation}
where $N$ is a polynomial in $u$ and $v$ which is given in Appendix \eqref{A9}. Therefore, $e_4=0$, which is a contradiction. Hence we have $\lambda=\eta=0$. We find that
\begin{align}\label{3.132}
\text{Res}(M,\lambda;u)\,&=v^6(1+v)^4P_{11}(v)Q_{11}(v),\\
\label{3.133}
\text{Res}(M,\lambda;v)\,&=u^6(1+u)^4P_{11}(u)Q_{11}(u),
\end{align}
where
\begin{equation}\label{3.134}
P_{11}(X)=X^{11}- X^{10}+X^8+X^7+X^5+X^4- X^3+X^2+X-1,
\end{equation}
\begin{equation}\label{3.135}
Q_{11}(X)=X^{11}- X^{10}+X^9+X^8+X^4- X^2+X+1
\end{equation}
are irreducible over $\f_3$. We already showed that $1+u\ne0$ and $1+v\ne 0$. Therefore $P_{11}(v)Q_{11}(v)=0$ and $P_{11}(u)Q_{11}(u)=0$. 
We claim that either $P_{11}(u)=Q_{11}(v)=0$ or $Q_{11}(u)=P_{11}(v)=0$.  In fact,
\begin{align*}
\text{Res}(P_{11}(u),M;u)=\,&-Q_{11}(v) (v^{11}- v^9- v^8- v^7- v^6+v^5+v^2-1)\cr
&\cdot (v^{11}- v^{10}+v^9- v^8+v^6+v^5-
   v^4+v^3+v^2- v-1),
\end{align*}
which is not divisible by $P_{11}(v)$. Hence $P_{11}(u)=0$ implies $Q_{11}(v)=0$. Similarly,
\begin{align*}
\text{Res}(Q_{11}(u),M;u)=\,&-P_{11}(v) (v^{11}+v^{10}- v^9+v^8+v^7+v^6+v^5+v^3-1)\cr
&\cdot (v^{11}- v^{10}- v^8+v^7-
   v^5+v^2+v+1) ,
\end{align*}
which is not divisible by $Q_{11}(v)$. Hence $Q_{11}(u)=0$ implies $P_{11}(v)=0$. Therefore the claim is proved.

First assume that $P_{11}(u)=Q_{11}(v)=0$. Solving the equation $Q_{11}(v)=0$ in $\f_3(u)$ by exhaustive search gives $v=v_0^{3^i}$, $0\le i\le 10$, where
\begin{equation}\label{v0}
v_0=-1+u^3+u^5-u^7+u^9-u^{10}.
\end{equation}
Thus $M|_{v=v_0^{3^i}}$ can be computed in $\f_3(u)$. We find that $M|_{v=v_0^{3^i}}=0$ only when $i=0$. Hence $v=v_0$. We compute $h_i|_{v=v_0}$ ($i=1,2$) and reduce the results modulo $P_{11}(u)$; the outcomes are
\begin{equation}\label{hibar}
h_i|_{v=v_0}\equiv\bar h_i\pmod{P_{11}(u)},\quad i=1,2,
\end{equation}
where $\bar h_1$ and $\bar h_2$ are polynomials in $u$ and $w$ given in \eqref{app-h1-v0} and \eqref{app-h2-v0}. We further compute
\begin{equation}\label{Ri}
R_i=\text{Res}(\bar h_i,L;u),\quad i=1,2,
\end{equation}
as polynomials in $w$, which are given in \eqref{app-res-h1-L} and \eqref{app-res-h2-L}. It follows that
\[
\text{gcd}(R_1,R_2)=(1+w)^2(1-w+w^3-w^7+w^8+w^{11}).
\]
We have $1+w\ne 0$ since $L|_{w=-1}=1+u^2-u^4+u^5\ne 0$.
Hence
\[
1-w+w^3-w^7+w^8+w^{11}=0.
\]
Solving the above equation (by exhaustive search) in $\f_3(u)$ gives $w=w_0^{3^i}$, $0\le i\le 10$, where
\[
w_0=1+u^3-u^4+u^6+u^9-u^{10}.
\]
Moreover, $L|_{w=w_0^{3^i}}=0$ only if $i=0$. Hence we must have $w=w_0$. Now we compute $e_4|_{v=v_0,w=w_0}$ as a polynomial in $u$, which turns out to be $\equiv 0\pmod{P_{11}(u)}$. Thus $e_4=0$, which is a contradiction.

Next assume that $Q_{11}(u)=P_{11}(v)=0$. The computation procedure is identical to the above. Since $M$ is symmetric in $u$ and $v$, the only solution $u_0$ of $Q_{11}(u)=0$ in $\f_3(v)$ satisfying $M|_{u=u_0}=0$ is the expression in \eqref{v0} with $u$ replaced by $v$, i.e.,
\begin{equation}\label{u0}
u_0=-1+v^3+v^5-v^7+v^9-v^{10}.
\end{equation}
We compute
\begin{equation}\label{hibar1}
h_i|_{u=u_0}\equiv\bar h_i\pmod{P_{11}(v)},\quad i=1,2,
\end{equation}
where $\bar h_1$ and $\bar h_2$ are polynomials in $v$ and $w$ given in \eqref{app-h1-u0} and \eqref{app-h2-u0}. Let $\bar L=\bar L(v,w)$ be the reduction of $L|_{u=u_0}$ modulo $P_{11}(v)$. Then
\begin{equation}\label{Ri-1}
R_i=\text{Res}(\bar h_i,\bar L;v),\quad i=1,2,
\end{equation}
are polynomials in $w$, which are given in \eqref{app-res-h1-Lbar} and \eqref{app-res-h2-Lbar}. Moreover,
\[
\text{gcd}(R_1,R_2)=w(1-w+w^2+w^3+w^4-w^5+w^7-w^8+w^9+w^{10}+w^{11}).
\]
Recall that $w\ne 0$; see \eqref{3.103}. Hence we have
\[
1-w+w^2+w^3+w^4-w^5+w^7-w^8+w^9+w^{10}+w^{11}=0.
\]
Solving the above equation (by exhaustive search) in $\f_3(v)$ gives $w=w_0^{3^i}$, $0\le i\le 10$, where
\[
w_0=-v^3+v^4-v^6-v^9+v^{10}.
\]
Moreover, $L|_{u=u_0,w=w_0^{3^i}}=0$ only if when $i=0$. Hence we must have $w=w_0$. We find that
\[
e_4|_{u=u_0,w=w_0}\equiv 0\pmod{P_{11}(v)}.
\]
Thus $e_4=0$, which is a contradiction.

\subsection*{Case 3.2} Assume that $v=0$. In this case we find that
\begin{equation}\label{3.136}
h_1=((1+u)^2+(1-w)^2)^2((u+w)^2+(1-w)^2)(u^2+w+w^2).
\end{equation}
In the above, we claim that $(1+u)^2+(1-w)^2\ne 0$ and $(u+w)^2+(1-w)^2\ne 0$. Otherwise, $w=1$ and $u=-1$. Then $1+u+v=0$, which is not possible. Therefore we have $u^2+w+w^2=0$; this is the first condition in \eqref{1.4}. It remains to show that
\[
1-\Bigl(\frac ba\Bigr)^{q+1}=1+\frac{w^2}{u^2}=\frac{-w}{u^2}
\]
is a square in $\f_q^*$, i.e., $-w$ is a square in $\f_q^*$. We find that
\[
e_4|_{v=0}=(1+u)^2(1+u+u^3+uw+uw^2)=(1+u)^3.
\]
Hence $1+u$ is a square in $\f_q^*$. Since
\[
(u+w)^2=u^2+w^2-uw=-w-uw=-w(1+u),
\]
we conclude that $-w$ is a square in $\f_q^*$.

Now the proof of Theorem~\ref{T3.1} is complete.

\section{Final Remarks}

The sufficiency of the condition \eqref{1.4} is actually implied by the proof in this paper. To see this, one only has to take a small portion of the proof and reverse the arguments there; we leave this task for interested readers.

The method of our proof is likely to work for a few more small characteristics. However, to prove that the conditions in Theorem~\ref{T1.1} (iii) are necessary for an arbitrary characteristic $p>3$, additional techniques might be needed.

\section*{Appendix}

In \eqref{3.14} and \eqref{3.15},
\begin{align}\label{A1}\tag{A1}
h_1=\,& a^6- a^5 b_1+a^4 b_1^2+a^4 b_1+a^4 k- a^4+a^3 b_1^3+a^3 b_1^2- a^3 b_1 k+a^3 b_1\\
&+a^3- a^2 b_1^4- a^2 b_1- a^2 k^2- a^2 k+a b_1^5+a b_1^4+a b_1^3 k- a b_1^3+a b_1^2\cr
&+a b_1 k^2+a b_1 k+a b_1- a k^2- a k- a+b_1^5- b_1^4 k+b_1^4- b_1^2 k^2\cr
&- b_1^2 k+b_1^2- b_1 k^2+b_1- k^3- k, \nonumber
\end{align}
\begin{align}\label{A2}\tag{A2}
h_2=\,& a^7- a^6 b_1+a^6 k+a^5 b_1+a^5- a^4 b_1^3- a^4 b_1^2- a^4 b_1 k+a^4 b_1- a^4 k\\
&+a^4+a^3 b_1^4- a^3 b_1^3 k+a^3 b_1^3- a^3 b_1^2 k+a^3 b_1 k+a^3 b_1- a^3 k^2+a^3 k\cr
&+a^3+a^2 b_1^4+a^2 b_1^3+a^2 b_1+a^2 k^2+a^2+a b_1^6- a b_1^5- a b_1^4 k+a b_1^4\cr
&- a b_1^3 k+a b_1^3- a b_1^2- a b_1 k^2- a b_1 k+a b_1- a k^2- a k- b_1^7\cr
&+b_1^6 k+b_1^6- b_1^5 k+b_1^4 k+b_1^4- b_1^3 k^2+b_1^3 k- b_1^3+b_1^2 k^2- b_1^2 k\cr
&- b_1 k^2+b_1 k- b_1- k^4+1. \nonumber
\end{align}
In \eqref{3.41} and \eqref{3.42},
\begin{align}\label{A3}\tag{A3}
h_1=\,& a_1^7 k^3+a_1^6 k^3- a_1^5 b_1 k^2- a_1^5 k^2- a_1^4 b_1^2 k^2+a_1^4 b_1 k^2- a_1^4 k^3- a_1^4 k^2\\
&+a_1^3 b_1^3 k+a_1^3 b_1^2 k^2- a_1^3 k^3- a_1^3 k^2+a_1^3 k+a_1^2 b_1^4 k- a_1^2 b_1 k^2+a_1^2 b_1 k\cr
&- a_1^2 k^2+a_1 b_1^6- a_1 b_1^5- a_1 b_1^4 k+a_1 b_1^4+a_1 b_1^3- a_1 b_1^2 k^2- a_1 b_1^2\cr
&+a_1 b_1 k^2- a_1 b_1 k+a_1 b_1+a_1 k^3- a_1 k^2- b_1^5+b_1^4 k+b_1^4+b_1^3 k\cr
&- b_1^3+b_1^2 k^2- b_1^2+b_1 k+b_1+k^3- k^2+k-1, \nonumber
\end{align}
\begin{align}\label{A4}\tag{A4}
h_2=\,& -a_1^9 k^4+a_1^5 b_1^3 k^2+a_1^5 b_1^2 k^2- a_1^5 b_1 k^2- a_1^5 k^2- a_1^4 b_1^2 k^2+a_1^4 b_1 k^2- a_1^4 k^2\\
&- a_1^3 b_1^6 k- a_1^3 b_1^5 k+a_1^3 b_1^4 k- a_1^3 b_1^3 k^2+a_1^3 b_1^2 k^2- a_1^3 b_1^2 k- a_1^3 b_1 k^2\cr
&+a_1^3 b_1 k+a_1^3 k- a_1^2 b_1^5 k+a_1^2 b_1^4 k+a_1^2 b_1^3 k^2- a_1^2 b_1^3 k+a_1^2 b_1^2 k^2- a_1^2 b_1^2 k\cr
&- a_1^2 b_1 k^2+a_1^2 b_1 k- a_1^2 k^2- a_1^2 k+a_1 b_1^7- a_1 b_1^5 k+a_1 b_1^4 k- a_1 b_1^4\cr
&- a_1 b_1^3 k- a_1 b_1^2 k^2- a_1 b_1^2 k+a_1 b_1 k^2+a_1 b_1 k+a_1 b_1- a_1 k^2- a_1 k\cr
&- b_1^7+b_1^6 k+b_1^6- b_1^5 k+b_1^4 k+b_1^4- b_1^3 k^2+b_1^3 k- b_1^3+b_1^2 k^2- b_1^2 k\cr
&- b_1 k^2+b_1 k- b_1- k^4+1. \nonumber
\end{align}
In \eqref{3.52},
\begin{align}\label{A6}\tag{A5}
T=\,& a_1^{10}- a_1^9 b_1+a_1^8 b_1^3+a_1^8 b_1^2-
   a_1^8 b_1- a_1^8+a_1^7 b_1^6- a_1^7
   b_1^5+a_1^7 b_1^4+a_1^7 b_1^3\\
   &+a_1^7 b_1^2-
   a_1^6 b_1^3+a_1^6 b_1+a_1^5 b_1^4- a_1^5
   b_1^3- a_1^5 b_1^2- a_1^5 b_1+a_1^5-
   a_1^4 b_1^9\cr
   &+a_1^4 b_1^8- a_1^4 b_1^6+a_1^4
   b_1^3+a_1^4 b_1+a_1^4+a_1^3 b_1^8+a_1^3
   b_1^7+a_1^3 b_1^6- a_1^3 b_1^5\cr
   &- a_1^3
   b_1^4- a_1^3 b_1^3+a_1^3 b_1^2+a_1^2
   b_1^7+a_1^2 b_1^3+a_1^2 b_1^2+a_1^2 b_1-
   a_1^2- a_1 b_1^5\cr
   &- a_1 b_1^4+a_1
   b_1^3+a_1 b_1^2+a_1 b_1- a_1+b_1^3-1. \nonumber
\end{align}
In \eqref{3.86} and \eqref{3.87},
\begin{align}\label{A7}\tag{A6}
h_1=\,& u^8+u^7 v- u^7 w+u^7+u^6 v w- u^6 w^2+u^5 v^3- u^5 v^2 w- u^5 v w^2\\
&- u^5 v w- u^5 v+u^5 w^2+u^5 w- u^4 v^3 w+u^4 v^3+u^4 v^2 w^2+u^4 v^2 w\cr
&- u^4 v^2- u^4 v w^2- u^4 v w+u^4 v+u^4 w^3- u^4 w^2- u^3 v^5- u^3 v^3\cr
&+u^3 v^2 w^2- u^3 v^2 w- u^3 v^2+u^3 v w^4+u^3 v w- u^3 w^4- u^3 w^3\cr
&- u^3 w^2- u^3 w- u^3+u^2 v^6- u^2 v^5 w- u^2 v^5- u^2 v^4 w^2- u^2 v^4 w\cr
&- u^2 v^4+u^2 v^3 w^3- u^2 v^3 w^2- u^2 v^3 w- u^2 v^3+u^2 v^2 w^3- u^2 v^2 w\cr
&- u^2 v^2- u^2 v w^2- u^2 v w- u^2 v+u^2 w^6- u^2 w^5- u^2 w^4- u^2 w^2\cr
&- u^2 w+u^2- u v^7- u v^6 w+u v^5 w^2+u v^5 w+u v^4 w^2+u v^4 w+u v^4\cr
&- u v^3 w^4- u v^3 w^3+u v^3 w^2+u v^3 w- u v^2 w^3+u v^2 w^2+u v^2 w\cr
&+u v w^5+u v w^4+u v w^3+u v w^2+u v w- u v- u w^7- u w^6+u w^5\cr
&+u w^3+u w^2- u w+v^6 w+v^5 w^3- v^5 w^2- v^4 w^4- v^4 w^3- v^4 w^2\cr
&- v^3 w^5+v^3 w^3- v^3 w+v^2 w^6- v^2 w^5- v^2 w^4+v^2 w^3- v^2 w^2\cr
&+v w^7+v w^6+v w^5- v w^3- v w^2- w^8+w^7+w^5- w^4+w^3+w, \nonumber
\end{align}
\begin{align}\label{A8}\tag{A7}
h_2=\,& u^7 v+u^7 w- u^7+u^6 v^3- u^6 v^2- u^6 v w+u^6 v+u^6 w^3- u^6 w\\
&+u^5 v^3- u^5 v^2 w+u^5 v w^2- u^5 v w- u^5 w^3+u^5- u^4 v^2 w^2\cr
&- u^4 v^2 w- u^4 v w^3- u^4 v w^2+u^4 v w+u^4 w^3- u^4 w- u^3 v^6\cr
&- u^3 v^5- u^3 v^3 w^3- u^3 v^3- u^3 v^2 w^3- u^3 v^2 w^2- u^3 v^2\cr
&- u^3 v w^3+u^3 v w^2+u^3 w^5+u^3 w^4- u^3 w^3- u^3 w^2+u^2 v^6\cr
&- u^2 v^5 w+u^2 v^5+u^2 v^4 w^2- u^2 v^4- u^2 v^3 w^3+u^2 v^3 w^2+u^2 v^3 w\cr
&- u^2 v^2 w^3- u^2 v^2 w+u^2 v^2- u^2 v w^4+u^2 v w^3+u^2 v w^2- u^2 v\cr
&- u^2 w^5- u^2 w^4+u^2 w^3+u^2 w^2+u^2 w- u^2- u v^7- u v^6 w\cr
&- u v^5 w^2+u v^5 w- u v^4 w^3+u v^4 w^2- u v^4 w+u v^4- u v^3 w^3\cr
&- u v^3 w^2- u v^3 w+u v^2 w^4- u v^2 w^2+u v^2 w+u v w^5- u v w^4\cr
&- u v w^3+u v w^2- u v w- u v+u w^5+u w^4- u w^3- u w^2\cr
&+v^7 w+v^6 w^3- v^6 w- v^5 w^3+v^4 w^3- v^4 w+v^3 w^5+v^3 w^3\cr
&+v^3 w- v^2 w^5- v^2 w^3+v w^5+v w^3+v w+w^9- w. \nonumber
\end{align}
In \eqref{res-u0},
\begin{align}\label{app-res-u0}\tag{A8}
&\text{Res}(h_1|_{u=u_0},h_2|_{u=u_0},w)\equiv\\
&v^{18}- v^{17}- v^{16}- v^{15}- v^{14}- v^{13}- v^{12}- v^{11}- v^{10}+v^9-
   v^8+v^6+v^3-1\cr
&\pmod{Q_{21}(v)}, \nonumber
\end{align}
\begin{align}\label{app-res-u0-}\tag{A9}
&\text{Res}(h_1|_{u=-u_0},h_2|_{u=-u_0},w)\equiv\\
&v^{20}- v^{18}+v^{17}- v^{15}+v^{14}+v^{12}+v^{10}+v^9- v^6- v^5+v^4- v^2+v \cr
&\pmod{Q_{21}(v)}. \nonumber
\end{align}
In \eqref{3.113} and \eqref{3.114},
\begin{align}\label{Pi1}\tag{A10}
&\Pi_1(w)=\\
& (w^8+w^7+w^5- w^4- w^3- w^2+w+1)^2 \cr
&\cdot(w^{16}+w^{15}- w^{13}+w^{12}+w^8+w^5- w^4+w^2- w-1)\cr
&\cdot (w^{18}+w^{16}- w^{15}+w^{14}- w^{13}- w^{12}- w^{11}- w^7- w^6- w^5- w^2- w+1), \nonumber
\end{align}
\begin{equation}\label{Pi2}\tag{A11}
\Pi_2(v)=(v^8+v^6+v^4- v^2-1)^2 (v^{16}- v^{15}- v^{13}+v^{11}+v^9+v^8- v^7+v^6+v^5+1).
\end{equation}
In \eqref{res-u1},
\begin{align}\label{app-res-u1}\tag{A12}
&\text{Res}(h_1|_{u=u_1},h_2|_{u=u_1},w)=\\
&v^{28} (v^3+v^2+v-1) (v^4+v^3- v^2- v-1) (v^5- v^4-
   v-1)\cr
&\cdot (v^{10}+v^9+v^8+v^6+v^5- v^4- v^2-1)\cr
&\cdot   (v^{12}+v^{11}+v^{10}- v^9+v^8- v^5- v^4+v^3+v+1)\cr
&\cdot (v^{46}-
   v^{45}- v^{44}- v^{43}- v^{42}- v^{41}+v^{40}- v^{39}- v^{38}- v^{36}-
   v^{35}+v^{34}+v^{33}- v^{32}\cr
   &\kern1em +v^{30}- v^{29}- v^{28}+v^{27}- v^{26}+v^{25}+v^{24}-
   v^{23}- v^{22}+v^{21}- v^{20}+v^{19}- v^{18}\cr
   &\kern1em -v^{15}+v^{14}+v^{12}+v^{11}+v^{10}+v^9- v^8+v^7+v^6- v^5+v^2+1)\cr
   &\cdot
   (v^{88}+v^{86}- v^{84}- v^{82}- v^{78}+v^{77}+v^{75}+v^{73}+v^{72}- v^{71}-
   v^{70}- v^{69}- v^{68}+v^{67}\cr
   &\kern1em  +v^{66}+v^{65}- v^{62}+v^{60}+v^{58}- v^{56}-
   v^{54}- v^{53}+v^{52}+v^{51}- v^{50}+v^{48}+v^{46}\cr
   &\kern1em  - v^{44}- v^{43}-
   v^{42}+v^{41}- v^{38}- v^{36}+v^{35}- v^{31}+v^{29}+v^{28}+v^{27}+v^{26}+v^{25}\cr
   &\kern1em  -v^{23}- v^{22}+v^{21}+v^{20}- v^{19}- v^{17}+v^{16}+v^{15}- v^{14}+v^{13}+v^{12}-
   v^{10}- v^8\cr
   &\kern1em  - v^6- v^5- v^4+v^3+1) , \nonumber
\end{align}
\begin{align}\label{app-res-u1-}\tag{A13}
&\text{Res}(h_1|_{u=-u_1},h_2|_{u=-u_1},w)=\\
&v^{31} (v^2+1)^4 (v^4+v^3+v^2+1) (v^{10}+v^8- v^7+v^6-
   v^5+v^4+v^3-1)\cr
   &\cdot (v^{18}- v^{15}- v^{14}- v^{12}+v^{10}+v^5- v^2-
   v+1)\cr
   &\cdot (v^{125}+v^{121}+v^{120}+v^{116}+v^{115}+v^{113}-
   v^{111}+v^{110}+v^{108}+v^{107}+v^{105}+v^{103}\cr
   &\kern1em  - v^{99}- v^{98}- v^{95}+v^{94}-
   v^{93}- v^{92}+v^{91}- v^{90}- v^{89}+v^{88}+v^{87}- v^{86}- v^{85}\cr
   &\kern1em  -v^{79}+v^{78}- v^{72}+v^{70}- v^{69}+v^{68}- v^{67}- v^{66}+v^{64}+v^{62}-
   v^{60}+v^{59}- v^{58}\cr
   &\kern1em  - v^{57}+v^{56}- v^{55}- v^{54}+v^{53}+v^{52}-
   v^{51}+v^{50}- v^{49}- v^{46}- v^{45}- v^{44}+v^{42}\cr
   &\kern1em  +v^{41}+v^{40}+v^{39}+v^{38}-
   v^{36}- v^{35}+v^{34}- v^{33}- v^{32}+v^{31}- v^{30}- v^{29}-
   v^{28}\cr
   &\kern1em  +v^{26}+v^{24}- v^{21}+v^{20}- v^{19}- v^{17}-
   v^{16}+v^{15}+v^{14}+v^{13}+v^{12}+v^6- v^5\cr
   &\kern1em  +v^4+v^3+v^2+v+1)  . \nonumber
\end{align}
The factors in \eqref{app-res-u1} and \eqref{app-res-u1-} are irreducible over $\f_3$.
In \eqref{3.120} and \eqref{3.121},
\begin{align}\label{Gamma1}\tag{A14}
&\Gamma_1(w)=\\
&-(w^4+w^3+w^2+1) (w^4+w^3+w^2- w-1)\cr
& (w^4- w^3+w^2+1)(w^8- w^6- w^4+w^3+1)\cr
& (w^{18}+w^{17}+w^{16}+w^{15}- w^{14}- w^{12}- w^8- w^7- w^6- w^3- w^2+1)\cr
& (w^{22}- w^{19}- w^{17}- w^{15}- w^{14}- w^{13}+w^{11}- w^8+w^5- w^4- w^3+w^2+w+1), \nonumber
\end{align}
\begin{align}\label{Gamma2}\tag{A15}
&\Gamma_2(v)=\\
& -(v^4+v^3-1) (v^4+v^3- v^2- v-1)^2 (v^8- v^7- v^6- v^5+v^4- v^2- v+1)\cr
& (v^{18}- v^{16}- v^{14}+v^{13}- v^{11}+v^9- v^5- v^4+v^3-1)\cr
& (v^{22}- v^{21}+v^{20}- v^{19}- v^{18}+v^{17}+v^{15}+v^{14}- v^{13}- v^{12}- v^{10}+v^9+v^7\cr
&+v^6- v^5+v^3+v^2+v+1). \nonumber
\end{align}
In \eqref{redp7q7}
\begin{align}\label{app-hij-1}\tag{A16}
h_{10}=\,& -v^6 w^4- v^6 w^3- v^6 w^2- v^6 w- v^5 w^4+v^5 w^3+v^5 w- v^5+v^4 w^4+v^4
   w^3 \\
   &- v^4 w^2+v^4 w- v^3 w^5+v^3 w^2- v^3 w+v^3+v^2 w^6- v^2 w^5+v^2 w^4- v^2
   w^2\cr
   &- v^2+v w^6- v w^5- v w^4- v w^3- v w^2- v w- v- w^6- w^5+w^4+w^3\cr
   &+w^2+1,  \cr
\label{app-hij-2}\tag{A17}
h_{11}=\,& -v^6 w^4+v^6 w^3+v^6 w- v^6- v^5 w^4- v^5 w^2+v^5- v^4 w^4+v^4 w^2+v^4 w+v^4\\
&- v^3
   w^4+v^3 w^3+v^3 w^2+v^3 w- v^2 w^4- v^2 w+v^2+v w^5- v w^4- v w^3- w^6,  \cr
\label{app-hij-3}\tag{A18}
h_{12}=\,& -v^6
   w^4+v^6 w^3- v^6 w^2+v^6 w- v^6+v^5 w^4- v^5 w^2+v^5 w+v^5- v^4 w^3\\
   &+v^4 w^2+v^4
   w+v^3 w^3+v^3+v^2 w^4+v^2 w^2+v^2 w+v^2+v w^4- v w^2+w^6\cr
   &- w^5+w^4+w,  \cr
\label{app-hij-4}\tag{A19}
h_{20}=\,& -v^6 w^2+v^6 w+v^6- v^5 w^5- v^5 w^4+v^5 w^2- v^5 w- v^5+v^4 w^5+v^4 w^4+v^4
   w^2\\
   &+v^4 w+v^4+v^3 w^5+v^3 w^3- v^3 w^2+v^3 w- v^2 w^5+v^2 w^2- v w^4- v w^3- v
   w\cr
   &+v- w^5+w^4- w^3 \cr
\label{app-hij-5}\tag{A20}
h_{21}=\,& v^6- v^5 w^5- v^5 w^4+v^5 w^3+v^5 w+v^4 w^5+v^4 w^4- v^4
   w^3+v^4 w^2+v^4 w- v^4\\
   &- v^3 w^2- v^3 w+v^2 w^4- v^2 w^3+v^2 w^2+v^2+v w^4- v
   w^3+w^5+w^4+w^3- w, \cr
\label{app-hij-6}\tag{A21}
h_{22}=\,& v^6 w^5+v^6 w^4+v^6 w^2- v^6 w- v^5 w^5- v^5 w^4- v^5 w^3-
   v^5 w^2- v^5 w+v^5+v^4 w^3\\
   &- v^4 w^2+v^4 w- v^4- v^3 w^3- v^3 w^2- v^3 w+v^2
   w^5+v^2 w^4+v^2 w^3- v^2 w^2+v^2\cr
   &+v w^5- v w^3+v w^2- v w- w^5- w^4- w^3- w^2-
   w-1. \nonumber
\end{align}
In \eqref{A0A1},
\begin{align}\label{app-A0}\tag{A22}
A_0=\,& v^6 w^6- v^6 w^5- v^6 w^2- v^6 w- v^6+v^5 w^6+v^5 w^4+v^5 w^3+v^5 w^2+v^5+v^4 w^6\\
&- v^4 w^5+v^4 w^4- v^4 w^3+v^4 w- v^4- v^3 w^5+v^3 w^4- v^3 w^3+v^3 w^2+v^3 w+v^3\cr
&- v^2 w^5- v^2 w^3+v^2 w^2- v^2- v w^6+v w^5+v w^4- v w^2+v+w^5- w^3- w, \cr
\label{app-A1}\tag{A23}
A_1=\,& -v^6 w^5+v^6 w^2+v^6- v^5 w^6+v^5 w^4- v^5 w^3- v^5 w^2+v^5 w- v^4 w^6- v^4 w^5\\
&- v^4 w^4- v^4 w^3+v^3 w^6- v^3 w^5- v^3 w^3+v^3 w^2+v^3 w- v^2 w^6+v^2 w^5+v^2 w^4\cr
&+v^2 w^2- v^2+v w^6+v w^5- v w^4- v w^3- v w^2- v w+v+w^5+w+1. \nonumber
\end{align}
In \eqref{U/V},
\begin{align}\label{app-U}\tag{A24}
U=\,& v^6 w^5+v^6 w^4+v^6 w^3- v^6 w^2+v^6 w+v^6+v^5 w^6- v^5 w^5- v^5 w^4+v^5- v^4 w^6\\
&+v^4 w^5+v^4 w^4- v^4 w^2- v^4 w+v^3 w^6- v^3 w^2+v^3 w- v^3- v^2 w^6- v^2 w^5\cr
&- v^2 w^4- v^2- v w^5- v w^2- w^6+w^5+w^4+w, \cr
\label{app-V}\tag{A25}
V=\,& v^6 w^6+v^6 w^4+v^6 w^3+v^6 w+v^5 w^5+v^5 w^4- v^5 w^3- v^5 w^2+v^5 w+v^5+v^4 w^6\\
&- v^4 w^5+v^4 w^4- v^4+v^3 w^6- v^3 w^5+v^3 w^4+v^3 w^3- v^3 w^2+v^3 w- v^2 w^6\cr
&+v^2 w^3+v^2 w^2+v^2- v w^4- v w^3+v w^2- v w+v+w^6- w^5+w^4- w^3\cr
&+w-1. \nonumber
\end{align}
In \eqref{redp21q21},
\begin{align}\label{app-Hij-1}\tag{A26}
h_{10}=\,& -w^4 v^{20}- w^2 v^{20}+v^{20}+w^4 v^{19}- w^3 v^{19}- w^2 v^{19}+w v^{19}- v^{19}+w^4
   v^{18}\\
   &+w^3 v^{18}- w v^{18}+w^4 v^{17}+w^3 v^{17}- w^2 v^{17}+w v^{17}- w^4 v^{16}+w^3 v^{16}-
   w v^{16}\cr
   &+v^{16}- w^4 v^{15}- w^3 v^{15}- w^2 v^{15}- w v^{15}+v^{15}- w^3 v^{14}- w v^{14}-
   w v^{13}+v^{13}\cr
   &- w^2 v^{12}+w^4 v^{11}- w^2 v^{11}- w v^{11}- w^4 v^{10}+w^3 v^{10}+w
   v^{10}+v^{10}- w^3 v^9\cr
   &- w^2 v^9- w v^9- w^4 v^8- w^2 v^8- v^8- w^4 v^7- w^3 v^7- w^2 v^7-
   w v^7+v^7\cr
   &+w^4 v^6- w^3 v^6+w v^6- v^6- w^4 v^5- w^3 v^5+w^2 v^5+v^5- w^4 v^4+w^3 v^4\cr
   &+w^2 v^4-
   v^4- w^5 v^3+w^4 v^3+w^3 v^3+w^2 v^3+w v^3+v^3+w^6 v^2- w^5 v^2\cr
   &- w^4 v^2- w^3 v^2- w^2 v^2-
   v^2+w^7 v+w^6 v+w^5 v+w^4 v- w^3 v- w v- w^8\cr
   &+w^7+w^5+w^3+w^2+1,  \cr
\label{app-Hij-2}\tag{A27}
h_{11}=\,& -w^4 v^{20}+w^2 v^{20}- w
   v^{20}+w^4 v^{19}+w^2 v^{19}+w^4 v^{18}- w^3 v^{18}+v^{18}+w^4 v^{17}\\
   &- w^3 v^{17}+w v^{17}-
   v^{17}- w^4 v^{16}- w^3 v^{16}- w^2 v^{16}- v^{16}- w^4 v^{15}- w^3 v^{15}\cr
   &- w
   v^{15}+v^{15}- w^2 v^{14}- w v^{14}+v^{14}- w^3 v^{13}+v^{13}+w^3 v^{12}+w v^{12}+w^4 v^{11}\cr
   &-w^3 v^{11}- w^2 v^{11}- w v^{11}+v^{11}- w^4 v^{10}+w^3 v^{10}- w^2 v^{10}- v^{10}- w^2 v^9\cr
   &+w
   v^9+v^9- w^4 v^8- w^3 v^8+w^2 v^8- v^8- w^4 v^7+w^3 v^7+v^7+w^4 v^6+w^3 v^6\cr
   &+w^2 v^6+w v^6- w^4
   v^5+w^3 v^5- w v^5- w v^4+v^4+w^3 v^3+w^2 v^3+v^3+w^2 v^2\cr
   &- w v^2+w^5 v+w^4 v- w^3 v+w^2 v+w v-
   v- w^7- w^6+w^5- w^4- w^3+w^2- 1, \cr
\label{app-Hij-3}\tag{A28}
h_{12}=\,& w^4 v^{20}+w^2 v^{20}- w v^{20}- v^{20}+w^4 v^{19}- w^3
   v^{19}- w^2 v^{19}- w v^{19}- v^{19}+w^4 v^{18}\\
   &- w^3 v^{18}- w^2 v^{18}- v^{18}+w^4 v^{17}-
   w^3 v^{17}- v^{17}- w^3 v^{16}- w^2 v^{16}- v^{16}+w^4 v^{15}\cr
   &- w^2 v^{15}- w v^{15}-
   v^{15}- w^4 v^{14}- w^3 v^{14}+w^2 v^{14}+v^{14}+w^4 v^{13}+w^3 v^{13}- w v^{13}\cr
   &- v^{13}- w^4
   v^{12}- w^3 v^{12}+w^2 v^{12}- v^{12}+w^3 v^{11}+w^2 v^{11}- v^{11}+w^4 v^{10}- w^2 v^{10}\cr
   &- w
   v^{10}+v^{10}- w^4 v^9- w^3 v^9- w^2 v^9- v^9- w^4 v^8+w^3 v^8+w^2 v^8- w^4 v^7\cr
   &+w^3 v^7- w^2
   v^7+w^3 v^6+w^2 v^6+w v^6+w^4 v^5- v^5- w^4 v^4- w^3 v^4- w^2 v^4\cr
   &+w v^4- w^3 v^3- w^2 v^3+w
   v^3- w^4 v^2- w^3 v^2- w^2 v^2- w v^2- v^2- w^4 v+w^2 v\cr
   &+w v+v+w^6- w^5+w^3+w+1,  \cr
\label{app-Hij-4}\tag{A29}
h_{20}=\,& - w^5 v^{20}- w^4 v^{20}+w^2 v^{20}+w v^{20}- w^3 v^{19}- w^2 v^{19}- v^{19}- w^3
   v^{18}+w^2 v^{18}\\
   &+w v^{18}- v^{18}+w^2 v^{17}+w v^{17}+w^5 v^{16}+w^4 v^{16}- w v^{16}- w^5
   v^{15}- w^4 v^{15}\cr
   &+w^3 v^{15}- v^{15}- w^5 v^{14}- w^4 v^{14}- w v^{14}+v^{14}+w^5 v^{13}+w^4
   v^{13}+w^3 v^{13}\cr
   &- w^2 v^{13}+w v^{13}- v^{13}- w^5 v^{12}- w^4 v^{12}+w^2 v^{12}- w
   v^{12}+v^{12}- w^5 v^{11}\cr
   &- w^4 v^{11}+w^2 v^{11}+w v^{11}- w^5 v^{10}- w^4 v^{10}+w^3 v^{10}-
   w^2 v^{10}- v^{10}- w^5 v^9\cr
   &- w^4 v^9+w^3 v^9+w v^9+w^3 v^8+w v^8- v^8- w^3 v^7- w^2 v^7+w
   v^7+v^7- w^5 v^6\cr
   &- w^4 v^6+w^3 v^6+w v^6- w^5 v^5- w^4 v^5+w^3 v^5- w^2 v^5+v^5- w^5 v^4- w^4
   v^4\cr
   &- w^3 v^4- w v^4- w^4 v^3- w^3 v^3+w^2 v^3- v^3- w^5 v^2+w^3 v^2+w v^2+v^2- w^5 v\cr
   &+w^4
   v+w^2 v- v+w^9+w^5+w^4- w^3- w^2+w+1,  \cr
\label{app-Hij-5}\tag{A30}
h_{21}=\,& - w^5 v^{20}- w^4 v^{20}+w^3 v^{20}+w^2 v^{20}+v^{20}+w^3
   v^{19}+w^2 v^{19}- v^{19}- v^{18}- w v^{17}\\
   &- v^{17}+w^5 v^{16}+w^4 v^{16}- w^3 v^{16}- w^2
   v^{16}+w v^{16}+v^{16}- w^5 v^{15}- w^4 v^{15}- v^{15}\cr
   &- w^5 v^{14}- w^4 v^{14}- w^3 v^{14}-
   w^2 v^{14}+v^{14}+w^5 v^{13}+w^4 v^{13}+w v^{13}+v^{13}\cr
   &- w^5 v^{12}- w^4 v^{12}+w v^{12}+v^{12}-
   w^5 v^{11}- w^4 v^{11}- w^3 v^{11}- w^2 v^{11}- v^{11}\cr
   &- w^5 v^{10}- w^4 v^{10}- w^3 v^{10}-
   w^2 v^{10}- w v^{10}- v^{10}- w^5 v^9- w^4 v^9- w^3 v^9\cr
   &- w^2 v^9+w v^9- v^9+w^3 v^8+w^2
   v^8- v^8- v^7- w^5 v^6- w^4 v^6+w^3 v^6+w^2 v^6\cr
   &- w v^6+v^6- w^5 v^5- w^4 v^5- w^3 v^5+w^2
   v^5- w v^5+v^5- w^5 v^4- w^4 v^4+w^3 v^4\cr
   &+w v^4+v^4- w^5 v^3- w^4 v^3- w^2 v^3- w v^3-
   v^3+w^4 v^2- w^2 v^2+v^2- w^5 v\cr
   &+w^2 v+w v+w^5+w^4- w^3+w^2- w- 1, \cr
\label{app-Hij-6}\tag{A31}
h_{22}=\,& - w^5 v^{20}- w^4 v^{20}- w^2
   v^{20}+w v^{20}- v^{20}- w^5 v^{19}- w^4 v^{19}- w^3 v^{19}- w^2 v^{19}\\
   &- w v^{19}+v^{19}-
   w^5 v^{18}- w^4 v^{18}- w^3 v^{18}- w^2 v^{18}+w v^{18}- v^{18}- w^5 v^{17}- w^4 v^{17}\cr
   &+w^3
   v^{17}- w^2 v^{17}- w v^{17}+w^3 v^{16}- w v^{16}- w^5 v^{15}- w^4 v^{15}- w^3 v^{15}- w^2
   v^{15}\cr
   &- w v^{15}+w^5 v^{14}+w^4 v^{14}+w^3 v^{14}+w^2 v^{14}- w v^{14}- w^5 v^{13}- w^4
   v^{13}+w^3 v^{13}\cr
   &- w^2 v^{13}+w^5 v^{12}+w^4 v^{12}+w^3 v^{12}+w^2 v^{12}- w v^{12}- w^3
   v^{11}+w v^{11}- w^5 v^{10}\cr
   &- w^4 v^{10}- w^3 v^{10}- w^2 v^{10}- w v^{10}+w^5 v^9+w^4 v^9+w^2
   v^9+v^9+w^5 v^8\cr
   &+w^4 v^8+w^3 v^8+w^2 v^8- w v^8+v^8+w^5 v^7+w^4 v^7+w^2 v^7+v^7- w^3 v^6- v^6\cr
   &-w^5 v^5- w^4 v^5+w^3 v^5- w^2 v^5+w v^5+v^5+w^5 v^4+w^4 v^4- w^3 v^4+v^4\cr
   &- w^2 v^3+v^3+w^3
   v^2+w^2 v^2- w v^2+v^2- w^4 v+w^3 v- w v+v+w^3+1.  \nonumber
\end{align}
In \eqref{A0A1-1},
\begin{align}\label{app-A0-1}\tag{A32}
A_0=\,& - w^{11} v^{20}- w^7 v^{20}+w^6 v^{20}+w^5 v^{20}+w^3 v^{20}+w^2 v^{20}- w v^{20}+w^{12} v^{19}+w^{11} v^{19}\\
&- w^{10} v^{19}- w^9 v^{19}+w^7 v^{19}- w^6 v^{19}- w^5 v^{19}- w^3 v^{19}- v^{19}+w^{12} v^{18}+w^{11} v^{18}\cr
&- w^9 v^{18}+w^7 v^{18}+w^6 v^{18}- w^5 v^{18}+w^3 v^{18}+w v^{18}- v^{18}+w^{12} v^{17}+w^{11} v^{17}\cr
&- w^{10} v^{17}+w^6 v^{17}- w^5 v^{17}- w^4 v^{17}+w^2 v^{17}- w v^{17}- v^{17}+w^{12} v^{16}- w^{11} v^{16}\cr
&- w^8 v^{16}- w^5 v^{16}- w^4 v^{16}- w^3 v^{16}+w^2 v^{16}- v^{16}- w^{11} v^{15}+w^{10} v^{15}- w^8 v^{15}\cr
&+w^7 v^{15}- w^6 v^{15}+w^4 v^{15}+w^3 v^{15}- w^2 v^{15}+w v^{15}- v^{15}+w^{12} v^{14}- w^{10} v^{14}\cr
&- w^9 v^{14}- w^8 v^{14}- w^7 v^{14}- w^5 v^{14}+w^4 v^{14}- w^3 v^{14}- w^2 v^{14}+w v^{14}- v^{14}\cr
&- w^{12} v^{13}- w^8 v^{13}- w^6 v^{13}+w^3 v^{13}+w^2 v^{13}- w v^{13}+v^{13}+w^{12} v^{12}- w^9 v^{12}\cr
&+w^6 v^{12}- w^5 v^{12}- w^2 v^{12}- w v^{12}- v^{12}- w^{12} v^{11}+w^{11} v^{11}+w^{10} v^{11}- w^8 v^{11}\cr
&+w^7 v^{11}- w^6 v^{11}+w^5 v^{11}- w^4 v^{11}- w^3 v^{11}- w^2 v^{11}+w v^{11}+v^{11}- w^{11} v^{10}\cr
&+w^8 v^{10}+w^6 v^{10}- w^4 v^{10}- w^3 v^{10}- w^2 v^{10}+v^{10}+w^{12} v^9+w^{10} v^9- w^9 v^9\cr
&- w^8 v^9- w^7 v^9- w^6 v^9- w^5 v^9- w^4 v^9- w^3 v^9- w^2 v^9- w v^9- v^9- w^{12} v^8\cr
&- w^{11} v^8+w^{10} v^8+w^8 v^8- w^7 v^8- w^6 v^8+w^5 v^8+w^4 v^8+w^3 v^8+w^2 v^8- w v^8\cr
&- v^8- w^{12} v^7- w^{11} v^7- w^{10} v^7- w^8 v^7+w^7 v^7- w^6 v^7+w^5 v^7- w^4 v^7- w^2 v^7\cr
&+v^7- w^{12} v^6+w^{11} v^6- w^{10} v^6- w^8 v^6- w^7 v^6- w^6 v^6+w^5 v^6+w v^6+v^6\cr
&- w^{11} v^5- w^9 v^5- w^8 v^5- w^4 v^5- w^2 v^5+w^{12} v^4+w^{11} v^4- w^9 v^4- w^8 v^4\cr
&- w^7 v^4+w^5 v^4+w^4 v^4- w^3 v^4+w v^4+w^{12} v^3- w^9 v^3+w^8 v^3- w^6 v^3+w^2 v^3\cr
&- w v^3+v^3+w^{13} v^2- w^{11} v^2- w^{10} v^2- w^8 v^2+w^7 v^2+w^5 v^2+w^4 v^2+w^3 v^2\cr
&+w v^2+w^{13} v- w^{12} v- w^{11} v+w^{10} v- w^9 v+w^7 v- w^6 v+w^5 v+w^3 v- w v\cr
&- w^{15}+w^{14}+w^{12}- w^{11}+w^{10}- w^8- w^7- w^6- w^3+w^2+w,  \cr
\label{app-A1-1}\tag{A33}
A_1=\,& w^{12} v^{20}- w^8 v^{20}- w^7 v^{20}- w^3 v^{20}+w^2 v^{20}- w v^{20}- v^{20}+w^{12} v^{19}- w^{11} v^{19}\\
&+w^{10} v^{19}+w^9 v^{19}- w^8 v^{19}+w^7 v^{19}+w^6 v^{19}+w^5 v^{19}+w^3 v^{19}- w^2 v^{19}- v^{19}\cr
&+w^{12} v^{18}- w^{11} v^{18}+w^{10} v^{18}- w^9 v^{18}- w^8 v^{18}+w^7 v^{18}- w^6 v^{18}+w^4 v^{18}\cr
&- w v^{18}+v^{18}+w^{12} v^{17}- w^{11} v^{17}- w^{10} v^{17}- w^8 v^{17}- w^7 v^{17}+w^6 v^{17}+w^4 v^{17}\cr
&+w^3 v^{17}+w^2 v^{17}- w v^{17}- v^{17}- w^{11} v^{16}- w^{10} v^{16}- w^8 v^{16}- w^5 v^{16}+w^4 v^{16}\cr
&+w^3 v^{16}+w^2 v^{16}+w^{12} v^{15}+w^{10} v^{15}- w^9 v^{15}+w^8 v^{15}+w^7 v^{15}- w^6 v^{15}\cr
&- w^5 v^{15}+w^4 v^{15}- w^3 v^{15}- w v^{15}- w^{12} v^{14}- w^{11} v^{14}- w^{10} v^{14}- w^9 v^{14}- w^7 v^{14}\cr
&- w^5 v^{14}+w^4 v^{14}- w^3 v^{14}- w v^{14}- v^{14}+w^{12} v^{13}+w^{11} v^{13}- w^{10} v^{13}+w^8 v^{13}\cr
&+w^7 v^{13}+w^4 v^{13}- w^3 v^{13}+w^2 v^{13}- v^{13}- w^{12} v^{12}- w^{11} v^{12}- w^{10} v^{12}+w^9 v^{12}\cr
&+w^8 v^{12}- w^7 v^{12}- w^5 v^{12}- w^4 v^{12}+w^3 v^{12}- w v^{12}+w^{11} v^{11}+w^{10} v^{11}- w^9 v^{11}\cr
&- w^8 v^{11}+w^6 v^{11}- w^5 v^{11}+w^2 v^{11}- w v^{11}- v^{11}+w^{12} v^{10}+w^{10} v^{10}+w^7 v^{10}\cr
&- w^6 v^{10}+w^5 v^{10}- w^4 v^{10}- w^2 v^{10}+v^{10}- w^{12} v^9- w^{11} v^9+w^7 v^9- w^6 v^9\cr
&- w^5 v^9- w^4 v^9+w^2 v^9+w v^9- v^9- w^{12} v^8+w^{11} v^8- w^{10} v^8- w^8 v^8- w^3 v^8\cr
&+w v^8- w^{12} v^7+w^{11} v^7- w^9 v^7+w^8 v^7+w^7 v^7+w^3 v^7+v^7+w^{11} v^6+w^{10} v^6\cr
&- w^8 v^6- w^7 v^6- w^5 v^6- w^3 v^6+w^2 v^6+w v^6+v^6+w^{12} v^5- w^{10} v^5+w^9 v^5\cr
&- w^8 v^5- w^7 v^5- w^6 v^5+w^4 v^5- w^2 v^5- v^5- w^{12} v^4- w^{11} v^4+w^{10} v^4+w^9 v^4\cr
&+w^6 v^4+w^5 v^4+w^4 v^4+w^{11} v^3+w^9 v^3- w v^3- v^3+w^9 v^2+w^8 v^2+w^6 v^2\cr
&+w^4 v^2- w^2 v^2- v^2- w^{11} v- w^8 v+w^6 v- w^4 v+w^3 v- w^2 v- w v- w^{11}\cr
&+w^{10}- w^9- w^8+w^7+w^6+w^4+w^3- w^2+w+1. \nonumber
\end{align}
In \eqref{U/V-1},
\begin{align}\label{app-U-1}\tag{A34}
U=\,& w^{20} v^{20}- w^{19} v^{20}+w^{18} v^{20}+w^{17} v^{20}- w^{14} v^{20}- w^{13} v^{20}+w^{12} v^{20}+w^{11} v^{20}\\
&- w^{10} v^{20}+w^9 v^{20}- w^8 v^{20}- w^6 v^{20}- w^5 v^{20}+w^3 v^{20}- w^{20} v^{19}- w^{17} v^{19}\cr
&+w^{16} v^{19}- w^{12} v^{19}+w^{11} v^{19}- w^9 v^{19}- w^8 v^{19}- w^6 v^{19}- w^5 v^{19}+w^4 v^{19}\cr
&- w^3 v^{19}+w^2 v^{19}+v^{19}- w^{19} v^{18}- w^{15} v^{18}+w^{14} v^{18}+w^{12} v^{18}- w^{11} v^{18}\cr
&+w^9 v^{18}- w^7 v^{18}- w^6 v^{18}+w^5 v^{18}+w^4 v^{18}- w^3 v^{18}+w v^{18}- w^{20} v^{17}- w^{17} v^{17}\cr
&+w^{15} v^{17}- w^{12} v^{17}- w^{10} v^{17}+w^9 v^{17}- w^7 v^{17}- w^6 v^{17}- w^5 v^{17}+w^4 v^{17}\cr
&+w^3 v^{17}+w^2 v^{17}+w v^{17}- v^{17}- w^{20} v^{16}+w^{19} v^{16}+w^{18} v^{16}+w^{17} v^{16}+w^{16} v^{16}\cr
&- w^{14} v^{16}+w^{13} v^{16}+w^{12} v^{16}+w^{11} v^{16}+w^{10} v^{16}+w^7 v^{16}+w^6 v^{16}+w^4 v^{16}\cr
&+w^2 v^{16}+w^{20} v^{15}+w^{19} v^{15}- w^{18} v^{15}+w^{17} v^{15}+w^{15} v^{15}+w^{14} v^{15}- w^{13} v^{15}\cr
&+w^{12} v^{15}- w^{10} v^{15}+w^8 v^{15}+w^5 v^{15}+w^3 v^{15}- w^2 v^{15}- v^{15}- w^{20} v^{14}+w^{18} v^{14}\cr
&+w^{17} v^{14}- w^{16} v^{14}+w^{15} v^{14}- w^{14} v^{14}+w^{11} v^{14}+w^9 v^{14}- w^7 v^{14}- w^6 v^{14}\cr
&+w^5 v^{14}+w^4 v^{14}- w^3 v^{14}+v^{14}+w^{20} v^{13}- w^{19} v^{13}- w^{17} v^{13}+w^{16} v^{13}+w^{15} v^{13}\cr
&- w^{14} v^{13}+w^{13} v^{13}- w^{11} v^{13}+w^9 v^{13}+w^8 v^{13}+w^6 v^{13}- w^5 v^{13}- w^3 v^{13}\cr
&+w^2 v^{13}- v^{13}- w^{20} v^{12}- w^{19} v^{12}- w^{17} v^{12}+w^{16} v^{12}+w^{14} v^{12}+w^{12} v^{12}\cr
&+w^9 v^{12}+w^8 v^{12}+w^7 v^{12}+w^6 v^{12}+w^5 v^{12}- w^2 v^{12}+w v^{12}- v^{12}- w^{20} v^{11}\cr
&+w^{19} v^{11}+w^{18} v^{11}+w^{16} v^{11}+w^{15} v^{11}+w^{12} v^{11}- w^{10} v^{11}- w^9 v^{11}+w^8 v^{11}\cr
&- w^7 v^{11}- w^6 v^{11}+w^5 v^{11}- w^4 v^{11}- w^3 v^{11}- w^2 v^{11}- v^{11}- w^{20} v^{10}+w^{19} v^{10}\cr
&+w^{16} v^{10}+w^{15} v^{10}- w^{14} v^{10}- w^{13} v^{10}- w^{12} v^{10}- w^{11} v^{10}+w^{10} v^{10}- w^9 v^{10}\cr
&- w^8 v^{10}+w^7 v^{10}+w^6 v^{10}+w^5 v^{10}+w^3 v^{10}+w^2 v^{10}+w v^{10}+v^{10}- w^{19} v^9\cr
&- w^{18} v^9+w^{16} v^9+w^{14} v^9+w^{13} v^9+w^{11} v^9+w^{10} v^9- w^9 v^9- w^7 v^9- w^6 v^9\cr
&- w^3 v^9- w^2 v^9+w v^9- w^{19} v^8+w^{17} v^8+w^{16} v^8- w^{15} v^8- w^{14} v^8- w^{13} v^8\cr
&+w^{12} v^8+w^{11} v^8- w^{10} v^8- w^9 v^8- w^8 v^8- w^4 v^8+w^2 v^8+w^{20} v^7+w^{19} v^7\cr
&+w^{17} v^7- w^{15} v^7+w^{14} v^7+w^{13} v^7- w^{11} v^7- w^9 v^7- w^8 v^7+w^7 v^7+w^4 v^7\cr
&+w^2 v^7- w v^7- v^7- w^{19} v^6- w^{17} v^6- w^{16} v^6+w^{15} v^6- w^{13} v^6- w^{12} v^6+w^{10} v^6\cr
&- w^6 v^6+w^5 v^6+w^3 v^6- v^6+w^{20} v^5+w^{18} v^5+w^{17} v^5- w^{15} v^5- w^{14} v^5- w^{12} v^5\cr
&- w^{11} v^5+w^{10} v^5- w^7 v^5+w^6 v^5- w^4 v^5+w^3 v^5+w^2 v^5+v^5- w^{19} v^4+w^{18} v^4\cr
&+w^{17} v^4- w^{15} v^4+w^{14} v^4+w^9 v^4+w^7 v^4- w^6 v^4+w^5 v^4- w^2 v^4- w v^4- v^4\cr
&- w^{20} v^3+w^{18} v^3+w^{15} v^3+w^{14} v^3- w^{13} v^3+w^{12} v^3+w^{10} v^3- w^9 v^3+w^7 v^3\cr
&+w^6 v^3- w^5 v^3+w^4 v^3- w^3 v^3- w^2 v^3- w v^3- w^{20} v^2+w^{19} v^2- w^{18} v^2- w^{17} v^2\cr
&+w^{12} v^2+w^{11} v^2- w^{10} v^2- w^9 v^2- w^4 v^2- w^3 v^2- w^2 v^2+w v^2+w^{19} v- w^{18} v\cr
&- w^{17} v- w^{16} v+w^{15} v- w^{14} v- w^{13} v- w^{11} v+w^{10} v- w^9 v+w^7 v+w^6 v\cr
&+w^5 v+w^2 v- v+w^{19}+w^{18}- w^{16}+w^{13}- w^{12}- w^9- w^8+w^7- w^6- w^5\cr
&+w^4- w^3+w^2-1,  \cr
\label{app-V-1}\tag{A35}
V=\,& w^{19} v^{20}- w^{18} v^{20}- w^{16} v^{20}+w^{15} v^{20}+w^{14} v^{20}+w^{13} v^{20}+w^{12} v^{20}+w^{11} v^{20}\\
&- w^{10} v^{20}+w^9 v^{20}- w^8 v^{20}- w^7 v^{20}- w^5 v^{20}+w^3 v^{20}- w^2 v^{20}- w v^{20}- w^{20} v^{19}\cr
&- w^{18} v^{19}+w^{17} v^{19}- w^{16} v^{19}- w^{15} v^{19}- w^{14} v^{19}- w^{13} v^{19}- w^{12} v^{19}- w^{11} v^{19}\cr
&+w^9 v^{19}- w^8 v^{19}+w^7 v^{19}+w^6 v^{19}+w^5 v^{19}- w^3 v^{19}+w^2 v^{19}+w v^{19}- v^{19}\cr
&- w^{17} v^{18}- w^{16} v^{18}+w^{14} v^{18}+w^{13} v^{18}- w^{10} v^{18}+w^9 v^{18}+w^8 v^{18}- w^5 v^{18}\cr
&+w^3 v^{18}+w^2 v^{18}- w v^{18}+v^{18}- w^{20} v^{17}- w^{18} v^{17}- w^{17} v^{17}- w^{16} v^{17}+w^{12} v^{17}\cr
&- w^{11} v^{17}+w^{10} v^{17}- w^8 v^{17}- w^6 v^{17}- w^5 v^{17}+w^4 v^{17}+w^3 v^{17}- w v^{17}+w^{20} v^{16}\cr
&- w^{19} v^{16}- w^{18} v^{16}+w^{15} v^{16}- w^{14} v^{16}- w^{12} v^{16}- w^{10} v^{16}+w^8 v^{16}+w^7 v^{16}\cr
&+w^6 v^{16}- w^5 v^{16}+w^4 v^{16}+w^3 v^{16}+w^2 v^{16}+w v^{16}+v^{16}- w^{18} v^{15}- w^{15} v^{15}\cr
&+w^{13} v^{15}+w^{11} v^{15}+w^{10} v^{15}- w^8 v^{15}- w^7 v^{15}- w^6 v^{15}+w^5 v^{15}+w^3 v^{15}+w v^{15}\cr
&- v^{15}+w^{19} v^{14}+w^{15} v^{14}- w^{13} v^{14}- w^{11} v^{14}- w^9 v^{14}- w^7 v^{14}- w^6 v^{14}- w^4 v^{14}\cr
&+w^2 v^{14}- w v^{14}- w^{20} v^{13}+w^{19} v^{13}+w^{17} v^{13}+w^{16} v^{13}- w^{15} v^{13}+w^{14} v^{13}\cr
&+w^{13} v^{13}- w^{12} v^{13}+w^{11} v^{13}+w^{10} v^{13}- w^9 v^{13}- w^4 v^{13}+w^2 v^{13}+w v^{13}- v^{13}\cr
&+w^{20} v^{12}+w^{18} v^{12}- w^{17} v^{12}- w^{15} v^{12}- w^{13} v^{12}- w^9 v^{12}+w^8 v^{12}- w^7 v^{12}\cr
&- w^6 v^{12}- w^5 v^{12}- w^4 v^{12}- w^3 v^{12}- w v^{12}+v^{12}- w^{20} v^{11}+w^{19} v^{11}- w^{18} v^{11}\cr
&- w^{17} v^{11}- w^{16} v^{11}+w^{14} v^{11}- w^{12} v^{11}+w^9 v^{11}+w^8 v^{11}+w^7 v^{11}+w^6 v^{11}\cr
&- w^5 v^{11}+w^4 v^{11}- w^3 v^{11}- w v^{11}+v^{11}+w^{20} v^{10}- w^{19} v^{10}- w^{17} v^{10}+w^{16} v^{10}\cr
&+w^{13} v^{10}- w^{11} v^{10}- w^{10} v^{10}- w^9 v^{10}- w^8 v^{10}- w^6 v^{10}- w^5 v^{10}+w^4 v^{10}- w^3 v^{10}\cr
&- w^2 v^{10}- w v^{10}+v^{10}- w^{18} v^9- w^{17} v^9- w^{16} v^9- w^{15} v^9- w^{13} v^9- w^{12} v^9\cr
&+w^9 v^9- w^7 v^9+w^3 v^9+v^9+w^{20} v^8- w^{18} v^8+w^{17} v^8- w^{16} v^8+w^{15} v^8+w^{14} v^8\cr
&+w^{13} v^8+w^{12} v^8- w^{11} v^8- w^9 v^8- w^6 v^8+w^5 v^8- w^4 v^8+w^3 v^8- w^2 v^8+w^{20} v^7\cr
&+w^{19} v^7- w^{18} v^7+w^{17} v^7- w^{16} v^7- w^{15} v^7+w^{11} v^7- w^{10} v^7- w^7 v^7- w^5 v^7\cr
&+w^4 v^7+w^2 v^7+w v^7- w^{20} v^6- w^{18} v^6+w^{17} v^6- w^{16} v^6+w^{15} v^6- w^{13} v^6\cr
&- w^{11} v^6- w^8 v^6- w^7 v^6- w^6 v^6+w^5 v^6- w^4 v^6+w^3 v^6+w^2 v^6+w v^6+v^6\cr
&- w^{19} v^5+w^{18} v^5- w^{17} v^5- w^{15} v^5- w^{14} v^5+w^{12} v^5- w^{11} v^5+w^{10} v^5- w^9 v^5\cr
&- w^8 v^5- w^6 v^5- w^5 v^5+w^4 v^5- w^3 v^5- w^2 v^5- w v^5- v^5+w^{19} v^4+w^{18} v^4\cr
&- w^{17} v^4+w^{16} v^4+w^{15} v^4+w^{14} v^4- w^{12} v^4+w^{10} v^4+w^8 v^4- w^7 v^4- w^6 v^4\cr
&- w^5 v^4+w^4 v^4+w^2 v^4- w v^4+w^{20} v^3- w^{19} v^3+w^{18} v^3+w^{14} v^3+w^{12} v^3\cr
&- w^{11} v^3- w^9 v^3+w^8 v^3- w^7 v^3+w^6 v^3+w^4 v^3+w^3 v^3- w v^3+v^3+w^{18} v^2\cr
&- w^{17} v^2+w^{16} v^2- w^{11} v^2+w^9 v^2- w^8 v^2- w^7 v^2+w^4 v^2+w^3 v^2- w v^2- v^2\cr
&- w^{20} v- w^{19} v- w^{18} v+w^{17} v- w^{16} v+w^{14} v+w^{13} v+w^{12} v+w^{11} v- w^9 v\cr
&- w^8 v- w^7 v- w^5 v+w^4 v- w^3 v+w^2 v- w v+v- w^{20}- w^{19}- w^{18}+w^{17}\cr
&- w^{15}- w^{14}- w^{12}+w^{10}- w^8+w^6- w^5+w^4- w^2+w- 1. \nonumber
\end{align}
In \eqref{3.131},
\begin{align}\label{A9}\tag{A36}
N=\,& v^6 u^{16}- v^7 u^{15}+v^8 u^{14}- v^6 u^{14}- v^5 u^{14}+v^5 u^{13}+v^4 u^{13}- v^3
   u^{13}\\
   &+v^{10} u^{12}+v^8 u^{12}- v^7 u^{12}- v^6 u^{12}- v^5 u^{12}- v^4
   u^{12}+v^3 u^{12}\cr
   &- v^{11} u^{11}- v^9 u^{11}- v^8 u^{11}+v^7 u^{11}+v^6 u^{11}+v^4
   u^{11}- v^3 u^{11}\cr
   &- v^2 u^{11}+v^{12} u^{10}+v^{10} u^{10}- v^9 u^{10}+v^8
   u^{10}+v^5 u^{10}+v u^{10}\cr
   &- u^{10}- v^{11} u^9- v^{10} u^9+v^9 u^9- v^8 u^9+v^7
   u^9- v^6 u^9+v^5 u^9\cr
   &- v^4 u^9- v^3 u^9+v^2 u^9- v u^9+u^9+v^{14} u^8+v^{12} u^8-
   v^{11} u^8\cr
   &+v^{10} u^8- v^9 u^8- v^8 u^8+v^7 u^8- v^6 u^8- v^5 u^8- v^4 u^8- v^3
   u^8\cr
   &- v^2 u^8- u^8- v^{15} u^7- v^{12} u^7+v^{11} u^7+v^9 u^7+v^8 u^7- v^6 u^7\cr
   &+v^5
   u^7- v^3 u^7+v^2 u^7- v u^7+v^{16} u^6- v^{14} u^6- v^{12} u^6+v^{11} u^6\cr
   &- v^9
   u^6- v^8 u^6- v^7 u^6- v^6 u^6- v^5 u^6- v^4 u^6+v^2 u^6- v^{14} u^5\cr
   &+v^{13}
   u^5- v^{12} u^5+v^{10} u^5+v^9 u^5- v^8 u^5+v^7 u^5- v^6 u^5+v^5 u^5\cr
   &+v^4 u^5+v^{13}
   u^4- v^{12} u^4+v^{11} u^4- v^9 u^4- v^8 u^4- v^6 u^4+v^5 u^4\cr
   &+v^4 u^4- v^{13}
   u^3+v^{12} u^3- v^{11} u^3- v^9 u^3- v^8 u^3- v^7 u^3- v^{11} u^2\cr
   &+v^9 u^2- v^8
   u^2+v^7 u^2+v^6 u^2+v^{10} u- v^9 u- v^7 u- v^{10}+v^9- v^8.  \nonumber
\end{align}
In \eqref{hibar},
\begin{align}\label{app-h1-v0}\tag{A37}
\bar h_1=\,& -u^{10} w^7+u^{10} w^5- u^{10} w^4+u^{10} w+u^9 w^7- u^9 w^5+u^9 w^2+u^9 w+u^9+u^8 w^6\\
&- u^8
   w^4+u^8 w^3- u^7 w^7+u^7 w^6- u^7 w^3+u^7 w+u^7- u^6 w^6+u^6 w^4+u^6 w^3\cr
   &- u^6 w^2- u^6 w+u^5
   w^7+u^5 w^5- u^5 w^4+u^5 w^3+u^5 w^2+u^5- u^4 w^6- u^4 w^5\cr
   &+u^4 w^3+u^4 w^2- u^4 w+u^3 w^7+u^3
   w^6- u^3 w^5+u^3 w^4+u^3 w^2- u^3 w- u^3\cr
   &+u^2 w^5- u^2 w^4- u^2 w^3+u^2 w^2- u^2 w- u w^7- u
   w^6+u w^5+u w^4- u w^2\cr
   &- u w- w^8- w^6- w^5- w^4+1,  \cr
\label{app-h2-v0}\tag{A38}
\bar h_2=\,& u^{10} w^5- u^{10} w^3- u^{10} w+u^{10}-
   u^9 w^5+u^9 w^4- u^9 w^3+u^9 w^2- u^9 w+u^8 w^5\\
   &- u^8 w^3- u^8 w- u^8- u^7 w^5- u^7 w^4- u^7
   w^2+u^7 w+u^6 w^5+u^6 w^3+u^6 w\cr
   &- u^5 w^5+u^5 w^4- u^5 w^3+u^5 w^2+u^4 w^5+u^4 w^4- u^4 w^3- u^4
   w- u^3 w^5\cr
   &+u^3 w^4+u^3 w^3+u^3 w^2- u^3 w+u^2 w^5- u^2 w^3+u^2 w^2- u^2+u w^5+u w^4\cr
   &+u w^2- u
   w+w^9+w^5- w^4- w^3+w^2- 1.\nonumber
\end{align}
In \eqref{Ri},
\begin{align}\label{app-res-h1-L}\tag{A39}
R_1=\,& -(w+1)^2 (w-1) (w^3+w^2- w+1) (w^7- w^6+w^5- w^4+w+1)\\
&\cdot (w^8-
   w^7+w^5- w^4- w^3+w-1) (w^{11}+w^8- w^7+w^3- w+1)\cr
   &\cdot (w^{17}+w^{16}-
   w^{13}+w^7- w^6- w^5+w^3+w^2- w+1)\cr
   &\cdot (w^{18}+w^{17}- w^{14}+w^{13}+w^{12}-
   w^{11}+w^{10}- w^9- w^7- w^6+w^4+w^3+w^2+1), \cr
\label{app-res-h2-L}\tag{A40}
R_2=\,& (w+1)^2 (w^{11}+w^8- w^7+w^3-
   w+1)\\
   &\cdot (w^{19}- w^{17}- w^{13}- w^{12}+w^{11}+w^9- w^8+w^6- w^4+w^2+1)\cr
   &\cdot
   (w^{31}- w^{29}- w^{28}+w^{27}- w^{25}- w^{22}- w^{21}- w^{19}+w^{17}+w^{16}- w^{15}\cr
   & \kern1em -w^{14}- w^{13}- w^{11}+w^{10}- w^9+w^8+w^7- w^6+w^4+w^2- w-1). \nonumber
\end{align}
In \eqref{hibar1},
\begin{align}\label{app-h1-u0}\tag{A41}
\bar h_1=\,&  v^{10} w^7- v^{10} w^6+v^{10} w^5- v^{10} w^4- v^{10} w^3- v^{10} w^2+v^{10} w- v^9
   w^7+v^9 w^6\\
   &- v^9 w^5- v^9 w^4- v^9 w^3+v^9 w^2+v^9 w+v^9+v^8 w^6- v^8 w^5+v^8 w^3- v^8 w^2\cr
   &+v^7
   w^7+v^7 w^5- v^7 w^4- v^6 w^6- v^6 w^5+v^6 w+v^6- v^5 w^7+v^5 w^6- v^5 w^4\cr
   &- v^5 w^3+v^5
   w^2+v^5 w+v^5- v^4 w^6+v^4 w^4+v^4 w^3+v^4 w- v^3 w^7- v^3 w^6\cr
   &- v^3 w^5+v^3 w^4+v^3 w^3- v^3
   w^2- v^3 w+v^3+v^2 w^5- v^2 w^4+v^2 w- v^2+v w^7\cr
   &+v w^6+v w^5- v w^3+v w^2- v w- v- w^8-
   w^7+w^6- w^5+w^4- w^3+w^2- w, \cr
\label{app-h2-u0}\tag{A42}
\bar h_2=\,& v^{10} w^5+v^{10} w^4- v^{10} w^2+v^{10} w- v^9 w^5+v^9 w^4- v^9
   w^3+v^9+v^8 w^5+v^8 w^4\\
   &- v^8 w^2+v^8 w- v^8- v^7 w^5+v^7 w^3- v^7 w^2- v^7 w- v^7+v^6 w^5+v^6
   w^4- v^6 w^3\cr
   &- v^6 w^2- v^5 w^5+v^5 w^4- v^5 w^3+v^5 w+v^4 w^5+v^4 w^3- v^4 w^2- v^4 w- v^3
   w^5\cr
   &+v^3 w^4+v^3 w^3- v^3+v^2 w^5+v^2 w^4+v^2 w^2+v^2 w+v w^5- v w^3+v w^2+v w\cr
   &+v+w^9+w^5- w^4-
   w^3+w^2.  \nonumber
\end{align}
In \eqref{Ri-1},
\begin{align}\label{app-res-h1-Lbar}\tag{A43}
R_1=\,& w (w^6+w^5+w^4+w^3- w+1)\\
&\cdot (w^{11}+w^{10}+w^9- w^8+w^7- w^5+w^4+w^3+w^2-
   w+1)\cr
   &\cdot (w^{21}- w^{20}- w^{19}+w^{18}- w^{17}+w^{16}+w^{15}- w^{13}- w^{12}-
   w^{11}+w^{10}- w^9\cr
   &\kern1em - w^7+w^6+w^5+w^4- w^3+w+1)\cr
   &\cdot (w^{71}+w^{70}+w^{68}- w^{67}+w^{66}-
   w^{65}- w^{64}- w^{63}+w^{61}+w^{59}- w^{57}- w^{56}\cr
   &\kern1em - w^{53}+w^{51}- w^{50}-
   w^{48}+w^{47}+w^{46}+w^{45}+w^{43}- w^{42}+w^{40}- w^{38}\cr
   &\kern1em - w^{37}- w^{36}- w^{35}+w^{33}-
   w^{32}+w^{29}- w^{27}+w^{26}- w^{25}- w^{24}+w^{21}\cr
   &\kern1em +w^{20}-
   w^{18}+w^{17}+w^{13}+w^{12}+w^{10}+w^9+w^8+w^7+w^6- w^5+w^3- w-1), \cr
\label{app-res-h2-Lbar}\tag{A44}
R_2=\,& w^2 (w^2-
   w-1) (w^{11}+w^{10}+w^9- w^8+w^7- w^5+w^4+w^3+w^2- w+1)\\
   &\cdot (w^{105}-
   w^{104}+w^{103}- w^{101}+w^{100}- w^{99}- w^{97}- w^{96}+w^{95}- w^{94}- w^{93}\cr
   &\kern1em  - w^{92}-
   w^{91}- w^{90}+w^{89}+w^{87}+w^{86}+w^{85}- w^{84}- w^{83}+w^{82}-
   w^{80}\cr
   &\kern1em +w^{79}+w^{78}+w^{77}+w^{76}- w^{74}- w^{73}- w^{72}+w^{70}+w^{69}+w^{68}+w^{67}\cr
   &\kern1em  -w^{66}+w^{65}+w^{64}+w^{63}- w^{62}+w^{60}- w^{59}+w^{58}+w^{57}- w^{56}- w^{55}\cr
   &\kern1em   -w^{54}+w^{53}- w^{51}- w^{50}+w^{47}+w^{44}+w^{43}- w^{42}+w^{38}- w^{37}- w^{36}\cr
   &\kern1em  +w^{32}-
   w^{31}+w^{29}- w^{28}+w^{27}- w^{23}+w^{22}+w^{21}- w^{20}- w^{19}- w^{18}\cr
   &\kern1em  -w^{17}+w^{13}+w^{11}+w^9- w^8+w^7- w^6- w^4- w^3- w^2- w-1).   \nonumber
\end{align}



\end{document}